\documentclass[12pt, a4paper]{article}
\usepackage{amsfonts}
\usepackage{mathrsfs}
\usepackage{tikz-cd}
\usepackage{latexsym}
\usepackage{longtable}
\usepackage{hyperref}
\usepackage{xy}
\usepackage{amsfonts,amsmath,amssymb,amsthm}
\usepackage{color}
\usepackage{bitset}

\xyoption{all}

\newcommand{\bcen}{\begin{center}}     \newcommand{\ecen}{\end{center}}
\newcommand{\bay}{\begin{array}}      \newcommand{\eay}{\end{array}}
\newcommand{\beq}{\begin{eqnarray*}}      \newcommand{\eeq}{\end{eqnarray*}}

\def\sup{\mathrm{sup}}
\def\inf{\mathrm{inf}}

\def\lim{\mathrm{lim}}

\def\Cone{\mathrm{Cone}}

\begin{document}
	
	\newtheorem{theorem}{Theorem}[section]
	\newtheorem{proposition}[theorem]{Proposition}
	\newtheorem{lemma}[theorem]{Lemma}
	\newtheorem{corollary}[theorem]{Corollary}
	\newtheorem{remark}[theorem]{Remark}
	\newtheorem{example}[theorem]{Example}
	\newtheorem{definition}[theorem]{Definition}
	\newtheorem{question}[theorem]{Question}
	\numberwithin{equation}{section}

	\title{\large\bf
 Quasi-projective dimension and Gorenstein projective dimension}
	
\author{\large Yongyun Qin$^{1,*}$ and Lanlan Yang$^1$}
\date{\footnotesize 1. School of Mathematics, Yunnan Key Laboratory of Modern Analytical Mathematics and Applications,
	Yunnan Normal University, Kunming, Yunnan 650500, China.
	\\ E-mail: qinyongyun2006@126.com; llyang2025@126.com
}
	
	\maketitle
	
	\begin{abstract} Gheibi, Jorgensen and Takahashi recently introduced the quasi-projective dimension, a homological invariant that extends the classical projective dimension. In this paper, we investigate this invariant from the perspective of Gorenstein homological algebra. First, we show that the quasi-projective dimension coincides with the Gorenstein projective dimension under certain conditions. Second, we introduce and study the quasi-Gorenstein projective dimension as a Gorenstein analogue of the quasi-projective dimension.

\end{abstract}
	\medskip
	
	
	\medskip
	
	
	\bigskip
	\section{\large Introduction}
	
	\indent\indent Projective and injective dimensions are fundamental invariants in classical homological algebra$.$ Recently, the concepts of quasi-projective and quasi-injective dimensions were introduced by Gheibi-Jorgensen-Takahashi \cite{GJT21} and Gheibi \cite{Gh24}, respectively. These quasi-homological dimensions not only generalize the classical homological dimensions, but also play a pivotal role in the Auslander-Buchsbaum formula and the depth formula \cite{CHY26, DFG, FL26, GJT21, Gh24, JMM26, JMM24}. Moreover, prominent conjectures such as the finitistic dimension conjecture, the Auslander-Reiten conjecture and Tachikawa's second conjecture are known to hold for Artin algebras whose finitely generated modules have finite quasi-projective dimension \cite{CCL26,GJT21}. This underscores the significance of determining the quasi-projective dimension.
	
	Recently, Chen-Chen-Liu developed some methods to compute the quasi-projective dimension \cite{CCL26}. Nevertheless, the computation remains intricate in general, even for a specific class of Nakayama algebras.
	In this paper, we establish a connection between
	the quasi-projective dimension $\operatorname{qpd}_{\mathcal{A}}M$ and the Gorenstein projective dimension $\mathrm{Gpd}_\mathcal{A} M$, where the latter is more readily computable.
	For the definitions of $\operatorname{qpd}_{\mathcal{A}}M$, we refer to Definition~\ref{def-quasi-proj}.
	
	\begin{theorem}  {\rm (Theorem~\ref{periodic})} \label{1.0} {\rm Let \(\mathcal{A}\) be an abelian category with enough projectives. If every Gorenstein projective object in $\mathcal{A}$ is periodic, then for any  $M\in\mathcal{A}$ with $\mathrm{Gpd}_\mathcal{A} M < \infty$, we have \(\mathrm{qpd}_\mathcal{A}M = \mathrm{Gpd}_\mathcal{A} M\).} 
	\end{theorem}
	
Theorem~\ref{1.0} offers an effective approach to calculating the quasi-projective dimension for certain modules over algebras whose Gorenstein projective modules are periodic. Such algebras include CM-finite algebras, and in particular, monomial algebras. Specifically, as every module over a Gorenstein algebra possesses finite Gorenstein projective dimension, the quasi-projective dimension of any module over CM-finite Gorenstein algebras and Gorenstein monomial algebras is computable. Consequently, the quasi-global dimension defined in \cite{CCL26} coincides with the Gorenstein global dimension for these algebras.

	A fundamental task in Gorenstein homological algebra is to investigate the Gorenstein counterparts of the classical homological algebra, such as the Gorenstein derived categories \cite{GZ10}, Gorenstein derived functors \cite{Hol04}, Gorenstein tilting modules \cite{MY, YLO, ZPei} and Gorenstein silting modules \cite{GMZ, YY}.	
In this paper, we define the quasi-Gorenstein projective dimension as a Gorenstein analogue of the quasi-projective dimension. Recall from \cite{GJT21} that the quasi-projective dimension of an object $M$ is defined via a quasi-projective resolution, which a complex of projective objects (not necessarily acyclic) whose homologies are isomorphic to finite direct sums of copies $M$. Motivated by this, we define the quasi-Gorenstein projective dimension in Definition~\ref{quasi} via such a complex $X_\bullet$ of Gorenstein projective objects,
with the extra condition that $\operatorname{H}_j(\operatorname{Hom}_{\mathcal{A}}(G, X_\bullet)) \cong \operatorname{Hom}_{\mathcal{A}}(G, \operatorname{H}_j(X_\bullet))$ for any
Gorenstein projective object $G$. This condition is crucial, because the functor $\operatorname{Hom}_{\mathcal{A}}(G,-)$ is required to be exact on the objects of interest in Gorenstein homological algebra. As a Gorenstein analogue of \cite{CCL26, GJT21}, 
we establish some elementary properties of quasi-Gorenstein projective dimension in Section~\ref{Quasi-Gorenstein-projective dimension }.

	This paper is organized as follows. In Section \ref{Section-definitions and conventions}, we recall some definitions and conventions. In section \ref{Quasi-projective dimension and Gorenstein projective dimension}, we establishes the relation between quasi-projective dimension and Gorenstein projective dimension, and we prove Theorem~\ref{1.0}. In section \ref{Quasi-Gorenstein-projective dimension }, we introduce and study the quasi-Gorenstein projective dimension.
	
	\section{\large Preliminaries}\label{Section-definitions and conventions}
	
	\indent\indent In this section, we introduce some standard notation and recall some basic facts about Gorenstein projective objects and Gorenstein derived categories.

	Throughout $\mathcal{A}$ is an abelian category with enough projective objects. 
	Let
	\[
	X_\bullet:\ \cdots \xrightarrow{d_{i+2}} X_{i+1} \xrightarrow{d_{i+1}} X_i \xrightarrow{d_i} X_{i-1} \xrightarrow{d_{i-1}} \cdots
	\]
	be a complex over $\mathcal{A}.$ For each integer $i$, we define the $i$-th cycle $\operatorname{Z}_i(X_\bullet) = \operatorname {Ker}d_i$ , the $i$-th boundary $\operatorname{B}_i(X_\bullet) = \operatorname{Im}d_{i+1}$, and the \(i\)-th homology $\operatorname{H}_i(X_\bullet) = \operatorname{Z}_i(X_\bullet) / \operatorname{B}_i(X_\bullet).$ For a fixed integer $n$, we denote by $X_\bullet[n]$ the complex obtained from $X_\bullet$ by shifting $n$ degrees, that is, $(X_\bullet[n])_i=X_{i-n}$. Moreover, the \textit{supremum}, \textit{infimum}, \textit{homological supremum} and \textit{homological infimum} of \(X_\bullet\) are defined by
	\[
	\operatorname{sup} X_\bullet := \operatorname{sup}\{i \in \mathbb{Z} \mid X_i \neq 0\}, \quad \operatorname{inf} X_\bullet := \inf\{i \in \mathbb{Z} \mid X_i \neq 0\},
	\]
	\[
	\operatorname{hsup} X_\bullet := \operatorname{sup}\{i \in \mathbb{Z} \mid \operatorname{H}_i(X_\bullet) \neq 0\}, \quad \operatorname{hinf} X_\bullet := \operatorname{inf}\{i \in \mathbb{Z} \mid \operatorname{H}_i(X_\bullet) \neq 0\}.
	\]
	Clearly $\inf X_\bullet \leqslant \operatorname{hinf} X_\bullet \leqslant \operatorname{hsup} X_\bullet \leqslant \sup X_\bullet.$ We say that $X_\bullet$ is \textit{bounded} (respectively, \textit{bounded below}) if $\sup X_\bullet < \infty$ and $\inf X_\bullet > -\infty$ (respectively, $\inf X_\bullet > -\infty$)$.$ 
	
	Let $\mathcal{C}(\mathcal{A})$ be the category of all complexes over $\mathcal{A}$ with chain maps, and let $\mathcal{C}^-(\mathcal{A})$ and $\mathcal{C}^b(\mathcal{A})$ the full subcategories of bounded below and bounded complexes, respectively. The corresponding homotopy categories are denoted by $\mathcal{K}(\mathcal{A})$, $\mathcal{K}^-(\mathcal{A})$ and $\mathcal{K}^b(\mathcal{A})$.
	
	Let $\mathcal{P}(\mathcal{A})$ be the class of projective objects of $\mathcal{A}$.  From \cite{EJ00}, an object \(G\) of \(\mathcal{A}\) is \textit{Gorenstein projective} if there is an exact sequence \(\dots \to P_1 \to P_0 \to P_{-1} \to P_{-2} \to \dots\) of projective objects of \(\mathcal{A}\), which stays exact after applying $\text{Hom}_{\mathcal{A}}(-,P)$ for each $P \in \mathcal{P(A)}$, such that $G \cong \text{Im}(P_0 \to P_{-1})$. Let \(\mathcal{GP(A)}\), or simply $\mathcal{GP}$, be the full subcategory of Gorenstein projective objects. For a finite dimensional algebra $A$ over a field $k$, we write  $\mathcal{P}(A)$ (resp. $\mathcal{GP}(A)$) for $\mathcal{P}(A\text{-mod})$ (resp. $\mathcal{GP}(A\text{-mod})$), where $A\text{-mod}$ is the category of finite generated left $A$-modules. 
	 
	 The \textit{Gorenstein projective dimension} $\operatorname{Gpd}_\mathcal{A}M$ of an object $M$ is defined to be the smallest integer $n\geqslant 0$ such that there is an exact sequence 
	 $0 \longrightarrow G_n \longrightarrow \cdots \longrightarrow G_0 \longrightarrow M \longrightarrow 0$
	 with $G_i \in \mathcal{GP(A)}$ for any $i$, if it exists; and \(\mathrm{Gpd} _\mathcal{A}M = \infty\) if there is no such exact sequence of finite length; see \cite{Hol}.
	  From \cite{BM}, the {\it Gorenstein global dimension} of $\mathcal{A}$ is defined as $$\mathrm{G. gldim}(\mathcal{A}) := \sup \{\operatorname{Gpd}_\mathcal{A}M \mid M \in \mathcal{A} \}.$$	  
	 
	  An Artin algebra $A$ is said to be a {\it Gorenstein algebra} provided the injective dimension of $_AA$ as well as of $A_A$ is finite. It is know that every module over a  Gorenstein algebra has finite Gorenstein projective dimension.  
	   
	  Following \cite{AM02}, we define a sequence to be {\it proper exact} if it remains exact under the functor $\operatorname{Hom}_{\mathcal{A}}(G, -)$ for any
	 $G\in \mathcal{GP(A)}$.  Since $\mathcal{P}(\mathcal{A})\subseteq
	 \mathcal{GP(A)}$, all the augmented proper exact sequence occurring in this paper will be exact. A \textit{proper Gorenstein projective resolution} of an object $M$ is a proper exact sequence $ \dots \to G_1\to G_0\to M\to 0$ with $G_i \in \mathcal{GP(A)}$ for any integer $i$. According to \cite{Hol}, any object with finite Gorenstein projective dimension admits a proper Gorenstein projective resolution of finite length. 
	
	 From \cite{GZ10}, the \textit{Gorenstein derived category} $\mathcal D_{\mathrm{gp}}^\ast(\mathcal{A})$  with \(\ast \in \{ \text{blank}, - ,b \}\) is defined as the Verdier quotient of the homotopy category $\mathcal{K}^\ast(\mathcal{A})$ modulo the thick triangulated subcategory $\mathcal{K}_\mathrm{gpac}^\ast(\mathcal{A})$ of proper exact
	 complexes (also known as $\mathcal{GP}$-acyclic complexes). A chain map $f_\bullet\colon X_\bullet \to Y_\bullet$ is called a {\it $\mathcal{GP}$-quasi-isomorphism} if it induces a quasi-isomorphism $\operatorname{Hom}_{\mathcal{A}}(G, f_\bullet)$ for each $G\in \mathcal{GP(A)}$; see \cite{GZ10}. It is known that the $\mathcal{GP}$-quasi-isomorphisms become isomorphisms in $\mathcal D_{\mathrm{gp}}^\ast(\mathcal{A})$.

\section{Quasi-projective dimension and Gorenstein projective dimension}\label{Quasi-projective dimension and Gorenstein projective dimension}

In this section, we recall the definition of quasi-projective dimension, and then establish its relation with Gorenstein projective dimension.
The following definition is taken from {\rm\cite[Definition 3.1]{GJT21}}. 
 
\begin{definition}\label{def-quasi-proj}
	
		{\rm(1) A \textit{quasi-projective resolution} of an object \(M\) in a category \(\mathcal{A}\) is defined as a bounded below complex \(P_\bullet\)
		of projective objects such that for all \(i \geq \inf P_\bullet\), there exist nonnegative integers \(a_i\), not all zero, such that \(H_i(P_\bullet) \cong M^{\oplus a_i}\). We say that the quasi-projective resolution \(P_\bullet\) is \textit{finite} if \(P_\bullet\) has finite length, equivalently, \(\sup P_\bullet < \infty\).
		
		(2) The \textit{quasi-projective dimension} of \(M\) is defined by
		$\operatorname{qpd}_{\mathcal{A}} M := 
			\inf\{\sup P_\bullet
			\linebreak
			  - \operatorname{hsup} P_\bullet \mid P_\bullet \text{ is a finite quasi-projective resolution of } M\}$, and 
			$\operatorname{qpd}_{\mathcal{A}} M =\infty$ if $M$ does not have a finite quasi-projective
			resolution.
		}
\end{definition}

Since each deleted projective resolution of $M$ is automatically a 
quasi-projective resolution, it is clear that $\operatorname{qpd}_{\mathcal{A}} M \leqslant \operatorname{pd}_{\mathcal{A}} M$, and it was proved in \cite[Proposition 3.1]{CCL26} that if $\operatorname{pd}_{\mathcal{A}} M<\infty$, then $\operatorname{qpd}_{\mathcal{A}} M = \operatorname{pd}_{\mathcal{A}} M$. However, for a
module $M$ with $\operatorname{pd}_{\mathcal{A}} M=\infty$, determining $\operatorname{qpd}_{\mathcal{A}} M$ is often difficult. The following theorem shows that, under certain conditions, the quasi-projective dimension can be characterized in terms of the Gorenstein projective dimension.

	 \begin{theorem}\label{periodic}
	 	Let \(\mathcal{A}\) be an abelian category such that every Gorenstein projective object in $\mathcal{A}$ is periodic. Then for any  $M\in\mathcal{A}$ with $\mathrm{Gpd}_\mathcal{A} M < \infty$, we have \(\mathrm{qpd}_\mathcal{A}M = \mathrm{Gpd}_\mathcal{A} M\).
	 \end{theorem}
	 \begin{proof}
	 	If \(\mathrm{Gpd}_\mathcal{A} M = 0\), then \(M \in \mathcal{GP(A)}\), and hence 
	 	$M$ is periodic by hypothesis. Now it follows from {\rm\cite[Proposition 3.6 (2)]{GJT21}} that \(\mathrm{qpd}_\mathcal{A} M = 0\).
	 	Assume that $\mathrm{Gpd}_\mathcal{A} M = d > 0.$ Then $\Omega^d M \in \mathcal{GP(A)}$, and hence there exists a positive integer $n$ such that $\Omega^n(\Omega^d M) \cong \Omega^d M.$  	
	Since $\Omega^{kn}(\Omega^d M) \cong \Omega^d M$ for any positive integer $k$, we may choose some $n \geq \mathrm{max}\{2,d\}$ such that $\Omega^{n+d} M \cong \Omega^d M$. Now we will apply {\rm\cite[Theorem 1.2 (4)]{CCL26}} to show \(\mathrm{qpd}_\mathcal{A} M \leqslant d\). 
	Consider the following chain map from
	the deleted projective resolution of \(M\) to its $n$-th shift:
	 	\[\xymatrix@C=11pt{
	 		\cdots\ar[r]^{f_1}&  Q_0\ar[r]^{f_n} \ar@{=}[d]& Q_{n-1}\ar[r]^{f_{n-1}} \ar@{-->}[d]^{h_{n-1}}&Q_{n-2}\ar[r] \ar@{-->}[d]^{h_{n-2}}&\cdots\ar[r]&Q_{n-d}\ar[r]\ar@{-->}[d]^{h_{n-d}}\ar[r]&\cdots\ar[r]^{f_1}&Q_0\ar[r]^{f_0}&P_{d-1}\ar[r]^{g_{d-1}}&\cdots\ar[r]^{g_1}&P_0\ar[r]&0\\
	 		\cdots\ar[r]^{f_1}& Q_0\ar[r]^{f_0}& P_{d-1}\ar[r]^{g_{d-1}}&P_{d-2}\ar[r]&\cdots\ar[r]&P_0\ar[r]&0,}
	 		\]
	 		where the morphisms \(h_{n-1}, h_{n-2}, \ldots, h_{n-d}\) exist for the following reasons.
	 		Since \(\Omega^d M \in \mathcal{GP(A)}\), it follows that 
	 		\(\mathrm{Ext}_\mathcal{A}^i(\Omega^d M, P) = 0\) for all \(P \in \mathcal{P(A)}\) and all \(i > 0\). 
	 		Since $$\cdots \to Q_1 \xrightarrow{f_1} Q_0 \xrightarrow{f_n} Q_{n-1} \xrightarrow{f_{n-1}} Q_{n-2} \to \cdots \to Q_1 \xrightarrow{f_1} Q_0 \to \Omega^d M \to 0$$
	 		is a projective resolution of \(\Omega^d M\), the following sequence
	 		\[
	 		\mathrm{Hom}_\mathcal{A}(Q_{n-1}, P_{d-1}) \xrightarrow{f_n^*} \mathrm{Hom}_\mathcal{A}(Q_0, P_{d-1}) \xrightarrow{f_1^*} \mathrm{Hom}_\mathcal{A}(Q_1, P_{d-1}) 
	 		\] is exact.
	 		Since \( f_1^*(f_0)= f_0 f_1 = 0\), we have $f_0 \in \mathrm{Ker} f_1^* = \mathrm{Im} f_n^*.$ 
	 		Hence, there exists a morphism \(h_{n-1} \in \mathrm{Hom}_\mathcal{A}(Q_{n-1}, P_{d-1})\) such that \(f_n^*(h_{n-1}) = f_0\), i.e., \(h_{n-1} f_n = f_0\). Similarly, consider the exact sequence $$\mathrm{Hom}_\mathcal{A}(Q_{n-2}, P_{d-2}) \xrightarrow{f_{n-1}^*} \mathrm{Hom}_\mathcal{A}(Q_{n-1}, P_{d-2}) \xrightarrow{f_n^*} \mathrm{Hom}_\mathcal{A}(Q_0, P_{d-2}).$$
	 		Since \(g_{d-1} h_{n-1} \in \mathrm{Hom}_\mathcal{A}(Q_{n-1}, P_{d-2})\) and 
	 		\(f_n^*(g_{d-1} h_{n-1}) = g_{d-1} h_{n-1} f_n = g_{d-1} f_0= 0\), 
	 		we have \(g_{d-1} h_{n-1} \in \mathrm{Ker} f_n^* = \mathrm{Im} f_{n-1}^*.\) 
	 		Then there exists \(h_{n-2} \in \mathrm{Hom}_\mathcal{A}(Q_{n-2}, P_{d-2})\) such that 
	 		\(g_{d-1} h_{n-1} = f_{n-1}^*(h_{n-2}) = h_{n-2} f_{n-1}.\) Proceeding inductively, we obtain morphisms \(h_{n-1}, h_{n-2}, \ldots, h_{n-d}\) making the above diagram commute.
	 		The above diagram induces an isomorphism $\Omega^{n+d}(M)\cong \mathrm{Coker}f_1\cong \Omega^d M$. Consequently, by {\rm\cite[Theorem 1.2 (4)]{CCL26}}, we obtain \(\mathrm{qpd}_\mathcal{A} M \leqslant d\). 
	 		It remains to prove that the equality is strict.
	 		
	 		Suppose, for contradiction, that \(\mathrm{qpd}_\mathcal{A} M = r < d\). By shifting, we assume that $M$ has a quasi-projective resolution of the form
	 		\[
	 		P_\bullet: 0 \to P_r \to \cdots \to P_1 \xrightarrow{d_1} P_0 \xrightarrow{d_0} P_{-1} \xrightarrow{d_{-1}} \cdots \to P_{-s} \to 0,
	 		\]
	 		where \(P_r\neq0 \), $\mathrm{hsup}(P_\bullet) = 0$, and for any $i \leqslant 0$,
	 		 there exist
	 		some integer $a_i$ (not all $a_i$ are zero) such that $\operatorname{H}_i(P_\bullet)\cong M^{\oplus a_i}$. Let $N= \operatorname {Coker}(d_1)$. Then it follows that $\operatorname{qpd}_\mathcal{A} M =r= \operatorname{pd}_\mathcal{A}N$. Following the proof of {\rm\cite[Proposition 2.2 ]{CCL26}}, we obtain the following short exact sequences:
	 		\begin{equation}\label{exact-sequence-1}
	 			0 \to M^{\oplus a_0} \to N \to \mathrm{Im}d_0 \to 0, 
	 		\end{equation}
	 		\begin{equation}\label{exact-sequence-2}
	 			0 \to \mathrm{Imd}_{i} \to \mathrm{Ker} d_{i-1} \to M^{\oplus a_{i-1}} \to 0, 
	 		\end{equation}
	 		\begin{equation}\label{exact-sequence-3}
	 			0 \to \mathrm{Ker} d_{i-1} \to P_{i-1} \to \mathrm{Im}d_{i-1} \to 0.
	 		\end{equation}
	 		Since \(\mathrm{Gpd}_\mathcal{A} M = d\), it follows from {\rm\cite[Theorem 2.20]{Hol}} that there exists some \(Q \in \mathcal{P(A)}\) such that \(\mathrm{Ext}_\mathcal{A}^d(M, Q) \neq 0\), and \(\mathrm{Ext}_\mathcal{A}^i(M, Q) = 0\) for all \(i > d\).
	 		As \(\mathrm{pd}_\mathcal{A} N = r < d\), applying \(\mathrm{Hom}_\mathcal{A}(-, Q)\) to  (\ref{exact-sequence-1}) yields
	 		$\mathrm{Ext}_\mathcal{A}^d(M^{\oplus a_0}, Q) \cong \mathrm{Ext}_\mathcal{A}^{d+1}(\mathrm{Im}\,d_0, Q).$
	 		Similarly, from (\ref{exact-sequence-2}) and (\ref{exact-sequence-3}), we obtain
	 		$\mathrm{Ext}_A^{d+1}(\mathrm{Ker} d_{-1}, Q)$  \(\cong \mathrm{Ext}_A^{d+1}(\mathrm{Im}d_0, Q)\)
	 		and  \(\mathrm{Ext}_\mathcal{A}^{d+1}(\mathrm{Ker} d_{-1}, Q) \cong \mathrm{Ext}_\mathcal{A}^{d+2}(\mathrm{Im}d_{-1}, Q)\). Consequently, we have a series of isomorphisms:
	 		\begin{center}
	 			$(\mathrm{Ext}_\mathcal{A}^d(M, Q))^{\oplus a_i}\cong\mathrm{Ext}_\mathcal{A}^d(M^{\oplus a_0}, Q) \cong \mathrm{Ext}_\mathcal{A}^{d+1}(\mathrm{Im}d_0, Q) \cong \mathrm{Ext}_\mathcal{A}^{d+2}(\mathrm{Im}d_{-1}, Q) \cong \cdots \cong \mathrm{Ext}_\mathcal{A}^{d+s+1}(\mathrm{Im}d_{-s}, Q)=0,$
	 		\end{center}
	 		where the last equality follows from the vanish of the map $d_{-s}\colon P_{-s}\to 0.$ Then we get \(\mathrm{Ext}_\mathcal{A}^d(M, Q) = 0\),  which is a contradiction. Therefore, we have \(\mathrm{qpd}_\mathcal{A} M = d=\operatorname{Gpd}_\mathcal{A}M\).
	 \end{proof}

	 \begin{remark}
	{\rm (1) If $\mathrm{Gpd}_\mathcal{A} M = \infty$,  then the equality \(\mathrm{qpd}_\mathcal{A} M = \operatorname{Gpd}_\mathcal{A}M\) does not hold. Indeed, in \cite[Example 3.7]{CCL26} (or Example~\ref{quiver1}), the simple $A$-module $S_2$ has the property that \(\mathrm{qpd}_{A} S_2 =0\) but $\operatorname{Gpd}_{A}S_2=\infty$; 
	
	(2) In Theorem~\ref{periodic}, the hypothesis that every Gorenstein projective object is periodic can not be omitted. Consider the finite-dimensional algebra 
	$R = k[x, y, u, v]/(x^2, xy, y^2, xu, yv, xv - yu, u^2, uv, v^2)$ with $k$ a field.  Then $R$ is a Gorenstein algebra, and thus every $R$-module has finite Gorenstein projective dimension. However,   
	it follows from \cite[Example 6.6]{GJT21} that there exists an $R$-module whose quasi-projective dimension is infinite.
	} 	
	 \end{remark}
	 
	Theorem~\ref{periodic} provides a method for computing the quasi-projective dimension for certain modules over algebras whose Gorenstein projective modules are periodic. Such algebras include CM-finite algebras, and in particular,
	monomial algebras. Indeed, an Artin algebra is called {\it CM-finite} if it has only finitely many isomorphism classes of indecomposable finitely generated Gorenstein projective
	modules. Since the syzygy functor $\Omega$ induces an auto-equivalence on the stable category of Gorenstein projective modules, for any CM-finite algebra 
	$A$ and any indecomposable module $M\in \mathcal{GP}(A)$, the set $\{\Omega ^iM \ |\  i \geq 0\}$ is finite. Consequently, $M$ is periodic, and so is every finitely generated Gorenstein projective $A$-module. By the results of \cite{CSZ18}, monomial algebras are always CM-finite.  
	 We now apply Theorem~\ref{periodic} to compute the quasi-projective dimension in cases where previous methods are computationally intensive.
	 
	 \begin{example}
	 	{\rm This is \cite[Example 3.10]{CCL26}.
	 		Let \(A\) be an algebra over a field \(k\) presented by the quiver
	 	\[
	 	\begin{tikzcd}
	 		1 \arrow[r, "\alpha"] 
	 		& 2 \arrow[d, "\beta"] \\
	 		4 \arrow[u, "\delta"] 
	 		& 3 \arrow[l, "\gamma"]
	 	\end{tikzcd}
	 	\]
	 	with relations: \(\delta\gamma\beta\alpha = \beta\alpha\delta\gamma = 0\), where the composition of arrows is always taken from right to left. Then \(A\) is a Gorenstein monomial algebra, and by \cite{CSZ18}, the indecomposable non-projective Gorenstein projective $A$-modules are precisely $M_1=A(\beta\alpha)$ and $ M_2=A(\delta\gamma).$
	 	Note that $M_1$ and $M_2$ are determined by their composition sequences:
	 	\[
	 	M_1 = \begin{array}{c} 3 \\ 4\end{array}, \ \ \ \ 	M_2 = \begin{array}{c} 1 \\ 2\end{array}.
	 	\]
	 	Denote by $J$ the Jacobson radical of $A$, and by $e_i$ the 
	 	primitive, idempotent element corresponding to the vertex $i$. Let $S_i$ be the simple $A$-module corresponding to $i$. It can be checked that $\mathrm{pd}_A(S_1)=\infty$.
	 	However, from the exact sequence
	 \[
	 \begin{tikzcd}[column sep=1em, row sep=0.2em]
	 	0 \arrow[r] 
	 	& M_2 \arrow[r] 
	 	& Ae_2 \arrow[rr] \arrow[rd, twoheadrightarrow] 
	 	&& Ae_1 \arrow[r] 
	 	& S_1 \arrow[r] 
	 	& 0, \\
	 	&&& \begin{array}{c} \small 2\\ \small 3\\ \small 4 \end{array} \arrow[ru, hook]
	 \end{tikzcd}
	 \]
	 	we deduce that $\mathrm{Gpd}_A(S_1)=2$, and hence $\mathrm{qpd}_A(S_1)=2$ by Theorem~\ref{periodic}. Similarly, the quasi-projective dimensions of all other indecomposable non-projective $A$-modules can be computed from their Gorenstein projective dimensions, which take values in $\{1, 2\}$.
	 	}
	 \end{example}
	 
	 	Next, we apply Theorem~\ref{periodic} to determine the quasi-projective dimension of certain modules over a non-Nakayama algebra.
	 \begin{example}
	 	{\rm Let \(A\) be an algebra over a field $k$ presented by the quiver
	 		\[
	 		\begin{tikzcd}[column sep=3pc]
	 			1 \arrow[r, bend left=25, "\alpha_1"] 
	 			& 2 \arrow[r, bend left=25, "\alpha_2"] \arrow[l, bend left=25, "\beta_1"] 
	 			& 3 \arrow[l, bend left=25, "\beta_2"] 
	 			& 4 \arrow[l, "\gamma"']
	 		\end{tikzcd}
	 		\]
	 		with relations: $\beta_1\alpha_1 = \alpha_1\beta_1=\beta_2\alpha_2 = \alpha_2\beta_2 = \beta_2\gamma=0$. Then \(A\) is a monomial algebra, and by \cite{CSZ18}, all indecomposable non-projective Gorenstein projective modules are $S_1$ and $Ae_3/Je_1=\begin{matrix} 2 \\ 3 \end{matrix}$. Now let $M = Ae_3/J^2e_3= \begin{matrix} 3 \\ 2 \end{matrix}$. Then from the exact sequence
	 		\[
	 		0 \to S_1 \to P_3 \to M \to 0,
	 		\]
	 		we deduce that $\operatorname{pd}_A M = \infty$
	 		but $\operatorname{Gpd}_A M = 1$, and thus $\operatorname{qpd}_A M = 1$ by Theorem~\ref{periodic}.}
	 	
	 \end{example}
	 	 
	 From \cite{CCL26}, the {\it quasi-global dimension} of an Artin algebra $A$ is defined as
	 \[
	 \mathrm{qgldim}(A) := \sup \{\mathrm{qpd}_A(M) \mid M \in A\text{-}\mathrm{mod}\}.
	 \]
	If $A$ is a Gorenstein algebra with periodic Gorenstein projective modules, then every $A$-module has finite Gorenstein projective dimension, and thus Theorem~\ref{periodic} implies that
	\[
	\mathrm{qgldim}(A) := \sup \{\mathrm{Gpd}_A(M) \mid M \in A\text{-}\mathrm{mod}\}=\mathrm{G.gldim}(A)=\mathrm{id}(_AA),
	\] where $\mathrm{id}(_AA)$ is the injective dimension of the regular module. This provides an effective method for computing the quasi-global dimension of CM-finite Gorenstein algebras, and in particular, Gorenstein monomial algebras.
	
	\begin{example}
		{\rm Let \(A_{n,m}\) be the Nakayama algebra in \cite[Theorem 4.8]{CCL26}, where $n\geqslant 2$
			and $1 \leqslant m \leqslant n$. The resolution quiver of \(A_{n,m}\) implies that \(A_{n,m}\) is a 
			Gorenstein algebra; see \cite[Proposition 1.2]{Shen}. By analyzing the injective resolution of the indecomposable projective modules, we have 	
			 \[
		\mathrm{qgldim}(A_{n,m})= \mathrm{id}(_{A_{n,m}}{A_{n,m}})= 
		\begin{cases} 
			2 & \text{for } 1 \leqslant m < n , \\
			0     & \text{for } m = n.
		\end{cases}
		\]	This give a brief proof of \cite[Theorem 4.8]{CCL26}.
		}
	\end{example}

	\section{Quasi-Gorenstein projective dimension}\label{Quasi-Gorenstein-projective dimension }
	
\indent\indent In this section, we will introduce the definition of quasi-Gorenstein projective dimension as a Gorenstein analogue of quasi-projective dimension, and then we give some basic properties and examples.

Recall from Definition~\ref{def-quasi-proj} that the quasi-projective dimension of an object $M$
is defined via a quasi-projective resolution, which is a complex of projective
objects (not necessarily acyclic) whose homologies are isomorphic to finite direct sums of copies $M$.
Since the functor $\operatorname{Hom}_{\mathcal{A}}(G,-)$ is required to be exact on the objects of interest in Gorenstein homological algebra
for any $G\in \mathcal{GP(A)}$,
we define the quasi-Gorenstein projective dimension via such a complex of Gorenstein projective objects,
with the extra condition that its homologies are preserved under the functor $\operatorname{Hom}_{\mathcal{A}}(G,-)$.
 The following lemma provides a justification for this condition.

\begin{lemma}\label{naturally}
	Let $\mathcal{A}$ and $\mathcal{B}$ be abelian categories, and let $F \colon \mathcal{A} \to \mathcal{B}$ be a left exact functor$.$ Then, for any $X_{\bullet}\in \mathcal{C}(\mathcal{A})$ and any $i\in \mathbb{Z}$, there is a naturally induced morphism $l_i \colon \operatorname{H}_i(F X_{\bullet}) \rightarrow F(\operatorname{H}_i(X_{\bullet})).$
	Moreover, if $l_i$ is an isomorphism for some $i$, then both $F\alpha_i$  and $F\pi_i$ are epimorphisms, where $\alpha_i\colon X_i \to \operatorname{B}_{i-1}(X_{\bullet})$ and 
	$\pi_i \colon \operatorname{Z}_i(X_{\bullet}) \to \operatorname{H}_i(X_{\bullet})$ are the canonical epimorphisms.
\end{lemma}
\begin{proof}
	Let $X_\bullet :\  \dots \to X_{i+1} \xrightarrow{d_{i+1}} X_i \xrightarrow{d_i} X_{i-1} \to \dots$ be a complex. Consider the following commutative diagram
		\[
		\begin{tikzcd}[column sep=small, row sep=large]
			& X_{i+1} \arrow[rr, "d_{i+1}"] \arrow[rd, twoheadrightarrow, "\alpha_{i+1}" description] 
			&& X_i \arrow[rr, "d_i"] 
			&& X_{i-1} \\
			\operatorname{Ker} d_{i+1} \arrow[ur, hook, "\omega_{i+1}" description]
			&& \operatorname{Im} d_{i+1} 
			\arrow[rr, hook, "\eta_i"] 
			\arrow[ur, hook, "\beta_{i+1}" description]
			&& \operatorname{Ker} d_i 
			\arrow[rr, twoheadrightarrow, "\pi_i"] 
			\arrow[ul, hook, "\omega_i" description]
			&& \operatorname{H}_i(X_\bullet),
		\end{tikzcd}
		\]	
	where $d_{i+1}=\beta_{i+1}\alpha_{i+1}$ is the standard decomposition, $\omega_i=\operatorname{Ker}d_i$, and $\eta_i$ is induced by the universal property of kernels. Indeed, since $d_i \beta_{i+1} \alpha_{i+1} = d_i d_{i+1} = 0$ and $\alpha_{i+1}$ is an epimorphism,  
	we have $d_i \beta_{i+1} = 0$. Consequently, there exists a morphism $\eta_i\colon\operatorname{Im} d_{i+1}\rightarrow \operatorname{Ker} d_i$ such that $\beta_{i+1} = w_i \eta_i.$ As $\beta_{i+1}$ is a monomorphism, it follows that $\eta_i$ is also a monomorphism. Set $\pi_i=\operatorname{Coker}\eta_i$, and then we have a short exact sequence
	$0 \to \operatorname{Im} d_{i+1} \xrightarrow{\eta_i} \ker d_i \xrightarrow{\pi_i} \operatorname{H}_i(X_\bullet) \to 0.$ Applying the left exact functor $F$, we have the following commutative diagram
		\[
		\begin{tikzcd}[column sep=0.14cm, row sep=1.0cm]
			& FX_{i+1} 
			\arrow[rr, "Fd_{i+1}"] 
			\arrow[rd, "F\alpha_{i+1}" description] 
			\arrow[ddr, twoheadrightarrow, bend right=15, "\mu_{i+1}" description]
			&& FX_i 
			\arrow[rr, "Fd_i"] 
			&& FX_{i-1} \\
			F(\operatorname{Ker} d_{i+1}) 
			\arrow[ur, hook, "F\omega_{i+1}"]
			&& F(\operatorname{Im} d_{i+1}) 
			\arrow[rr, hook, "F\eta_i" description, pos=0.55] 
			\arrow[ur, hook, "F\beta_{i+1}" description]
			&& F(\operatorname{Ker} d_i) 
			\arrow[rr, "F\pi_i"] 
			\arrow[ul, hook, "F\omega_i" description]
			&& F(\operatorname{H}_i(X_\bullet)), \\
			&& \operatorname{Im}(Fd_{i+1}) 
			\arrow[u, dashed, "g"] 
			\arrow[ruu, hook, bend right=15, "\nu_{i+1}" description, near start] 
			\arrow[rru, hook, dashed, bend right=15, "\theta_i" description]
		\end{tikzcd}
		\]
	where $Fd_{i+1}=\nu_{i+1}\mu_{i+1}$ is the standard decomposition.
	Then it follows from the fact $(Fd_{i+1})(Fd_{i})=0$ that $(Fd_{i})\nu_{i+1}=0$.
	Since $F$ is left exact, we have $F(\operatorname{Ker} d_i) \cong \operatorname{Ker} (Fd_i),$ and thus $F\omega_i=\operatorname{Ker}(Fd_i).$ 
	By the universal property of kernels, there exists a morphism
	$\theta_i : \operatorname{Im}(Fd_{i+1})\rightarrow F(\ker d_i)$ such that $\nu_{i+1} = (Fw_i)\theta_i.$ Because $\nu_{i+1}$ is a monomorphism, so is $\theta_i$.
	
	On the other hand, we have $F\omega_{i+1} \cong \operatorname{Ker}(Fd_{i+1})=\operatorname{Ker}(\nu_{i+1}\mu_{i+1})=\operatorname{Ker}\mu_{i+1}$. Since $\mu_{i+1}$ is an epimorphism,
	we have $\mu_{i+1} = \operatorname{Coker}(F\omega_{i+1})$. Moreover,
	it follows from the fact $(Fd_{i+1}) (F\omega_{i+1})= 0$ that  
	$(F\alpha_{i+1})(F\omega_{i+1}) = 0$, and then
	there exists a morphism $g\colon\operatorname{Im}(Fd_{i+1})\rightarrow F(\operatorname{Im} d_{i+1})$ such that $F\alpha_{i+1} = g  \mu_{i+1}.$ 
	Therefore, we have
	$(F\beta_{i+1}) g  \mu_{i+1} = (F\beta_{i+1})(F\alpha_{i+1}) = Fd_{i+1} = \nu_{i+1} \mu_{i+1},$ which yields $(F\beta_{i+1})g = \nu_{i+1}$ 
	since $\mu_{i+1}$ is an epimorphism. It follows that $(Fw_i)\theta_i = \nu_{i+1}= (F\beta_{i+1})g = (F\omega_i)(F\eta_i)g$, which implies $\theta_i = (F\eta_i)g$ since $F\omega_i$ is a monomorphism.
	Consider the following commutative diagram with exact rows:
	\[
	\begin{tikzcd}
		0 \arrow[r] 
		& \operatorname{Im}(Fd_{i+1}) \arrow[r, "\theta_i"] \arrow[d, "g"'] 
		& F(\ker d_i) \arrow[r, "\delta_i"] \arrow[d, equals] 
		& \operatorname{H}_i(FX_\bullet) \arrow[r] \arrow[d, dashed, "l_i"'] 
		& 0 \\
		0 \arrow[r] 
		& F\operatorname{Im} d_{i+1} \arrow[r, "F\eta_i"] 
		& F(\ker d_i) \arrow[r, "F\pi_i"] 
		& F(\operatorname{H}_i(X_\bullet)),
	\end{tikzcd}
	\]
	where $\delta_i$ is the cokernel of $\theta_i$.  
	Then there exists a morphism $l_i\colon \operatorname{H}_i(F X_\bullet) \rightarrow F(\operatorname{H}_i(X_\bullet))$ such that $F\pi_i = l_i \delta_i.$
	
	If $l_i$ is an isomorphism, then $F\pi_i=l_i \delta_i$ is an epimorphism. Moreover, the Snake Lemma implies that $g$ is an isomorphism, so $F\alpha_{i+1} = g  \mu_{i+1}$
	is an epimorphism.
\end{proof}

Next, we claim that the morphism $l_i$ in Lemma~\ref{naturally} is functorial.

\begin{lemma}\label{commutate}
	Let $F\colon \mathcal{A} \to \mathcal{B}$ be a left exact functor, and let $f_\bullet\colon X_\bullet \to Y_\bullet$ be a chain map between complexes. Then there exist two natural morphisms
	 $l_i^{X_\bullet} \colon \operatorname{H}_i(FX_\bullet) \to F(\operatorname{H}_i(X_\bullet))$ and $l_i^{Y_\bullet}\colon \operatorname{H}_i(FY_\bullet) \to F(\operatorname{H}_i(Y_\bullet))$ such that
	$ F(\operatorname{H}_i(f_\bullet))\circ l_i^{X_\bullet}= l_i^{Y_\bullet} \circ \operatorname{H}_i(Ff_\bullet)$.
\end{lemma}
\begin{proof}
	By Lemma \ref{naturally}, there are two morphisms $l_i^{X_\bullet}\colon \operatorname{H}_i(FX_\bullet) \to F(\operatorname{H}_i(X_\bullet))$ and $l_i^{Y_\bullet} \colon \operatorname{H}_i(FY_\bullet) \to F(\operatorname{H}_i(Y_\bullet))$
	such that $F(\pi_i^{X_\bullet}) = l_i^{X_\bullet} \delta_i^{X_\bullet}$ and $F(\pi_i^{Y_\bullet}) = l_i^{Y_\bullet} \delta_i^{Y_\bullet}.$
	Consider the following diagram 
	\[
	\begin{tikzcd}
		\operatorname{H}_i(FX_\bullet) 
		\arrow[rrr, "\operatorname{H}_i(Ff_\bullet)"] 
		\arrow[dd, "l_i^{X_\bullet}"]
		&&& \operatorname{H}_i(FY_\bullet) 
		\arrow[dd, "l_i^{Y_\bullet}"] \\
		& F(\operatorname{Z}_i(X_\bullet)) 
		\arrow[r, "F\bar{f}_i"] 
		\arrow[ul, twoheadrightarrow, "\delta_i^{X_\bullet}" description] 
		\arrow[dl, "F(\pi_i^{X_\bullet})"]
		& F(\operatorname{Z}_i(Y_\bullet)) 
		\arrow[ur, twoheadrightarrow, "\delta_i^{Y_\bullet}" description] 
		\arrow[dr, "F(\pi_i^{Y_\bullet})"] \\
		F(\operatorname{H}_i(X_\bullet)) 
		\arrow[rrr, "F(\operatorname{H}_i(f_\bullet))"]
		&&& F(\operatorname{H}_i(Y_\bullet)),
	\end{tikzcd}
	\]
	where $\bar{f}_i \colon \operatorname{Z}_i(X_\bullet) \to \operatorname{Z}_i(Y_\bullet)$ is induced by $f_i\colon X_i \to Y_i.$ By the construction of
	 $\operatorname{H}_i(f_\bullet)$ and $\operatorname{H}_i(Ff_\bullet)$,
	 the part of the above diagram involving
	 $F\bar{f}_i$ commutes. This yields $F(\operatorname{H}_i(f_\bullet)) \circ l_i^{X_\bullet} = l_i^{Y_\bullet} \circ \operatorname{H}_i(Ff_\bullet)$.
\end{proof}

The morphism $l_i$
in Lemma~\ref{naturally} is referred to as the natural morphism induced by $F$.

\begin{definition} \label{quasi}{\rm(1) A complex $X_{\bullet} \in \mathcal{C}^{-}(\mathcal{A})$ is called a {\it quasi-Gorenstein projective resolution} of an object $M$ if the following three conditions hold:

(i) all $X_i$ are Gorenstein projective for $i \in \mathbb{Z}$;

(ii) for all integers $j \geqslant \inf \left(X_{\bullet}\right),$ there are integers $a_j \geqslant 0,$ not all zero, such that $\operatorname{H}_j\left(X_{\bullet}\right) \cong M^{\oplus a_j}$;

(iii) for any \(G \in \mathcal{GP}(\mathcal{A})\) and every $i \in \mathbb{Z}$, there is a natural isomorphism $ \operatorname{H}_i(\operatorname{Hom}_{\mathcal{A}}(G, X_{\bullet})) \cong \operatorname{Hom}_{\mathcal{A}}(G, \operatorname{H}_i(X_{\bullet}))$ induced by $\operatorname{Hom}_{\mathcal{A}}(G, -)$.

(2) The {\it quasi-Gorenstein projective dimension} of $M$ is defined to be
\[
\operatorname{qGpd}_{\mathcal{A}} M := \inf \left\{ \operatorname{sup}X_{\bullet} - \operatorname{hsup}X_{\bullet} \;\middle|\; \parbox{0.45\textwidth}{$X_{\bullet}$ is a finite quasi-Gorenstein projective resolution of $M$.} \right\}.
\] }
\end{definition}
We understand that $\operatorname{qGpd}_{\mathcal{A}} M =\infty$ if $M$ does not admit a finite quasi-Gorenstein projective resolution,
and  $\operatorname{qGpd}_{\mathcal{A}} M =0$ if $M=0$.
 According to {\rm\cite[Theorem~2.10]{Hol}}, if $\operatorname{Gpd}_{\mathcal{A}}M<\infty$, then $M$ admits a proper Gorenstein projective resolution of finite length. Since each deleted proper Gorenstein projective resolution is a 
 quasi-Gorenstein projective resolution, it is clear that $\operatorname{qGpd}_{\mathcal{A}}M \leqslant \operatorname{Gpd}_{\mathcal{A}}M.$
 The inequality can be strict (see Example~\ref{quiver1}). If $\mathcal{GP}(\mathcal{A}) = \mathcal{P}(\mathcal{A})$, then the quasi-Gorenstein projective dimension coincides with the quasi-projective dimension.

We now establish some fundamental properties of the quasi-Gorenstein projective dimension. The proof is analogous to those in \cite{CCL26, GJT21}, but requires additional steps to verify whether the homologies of the complexes are preserved by the functor $\operatorname{Hom}_{\mathcal{A}}(G, -)$.
The following result constitutes the Gorenstein analogue of
\cite[Proposition 3.3]{GJT21} and \cite[Proposition 3.6 (2)]{GJT21}.

\begin{proposition} \label{basis}
{\rm(1)} For an object $M \in \mathcal{A}$ and an integer $n>0$, one has that $\operatorname{qGpd}_{\mathcal{A}}\left(M^{\oplus n}\right)=\operatorname{qGpd}_{\mathcal{A}}M$.

{\rm(2)} Let $M, N \in \mathcal{A}$. Then $\operatorname{qGpd}_{\mathcal{A}}(M \oplus N) \leqslant \sup \left\{\operatorname{qGpd}_{\mathcal{A}}M, \operatorname{qGpd}_{\mathcal{A}}N\right\}$. 

{\rm(3)} For a nonzero object $M \in \mathcal{A}$ and a Gorenstein projective object $G \in \mathcal{A}$, one has that $\operatorname{qGpd}_{\mathcal{A}}(M \oplus G) \leqslant \operatorname{qGpd}_{\mathcal{A}}M$.

{\rm(4)} If $M$ is Gorenstein periodic, that is, there is a proper
Gorenstein projective resolution  $\cdots\rightarrow G_{n}\xrightarrow{d_n} G_{n-1}\rightarrow \cdots \rightarrow G_1 \xrightarrow{d_1} G_0\xrightarrow{d_0} M \rightarrow 0$ such that $\operatorname{Im}d_{r}\cong M$ for some $r\geqslant 1$,
then $\operatorname {qGpd}_\mathcal{A}M = 0$.
\end{proposition}

\begin{proof}
(1) If $X_\bullet$ is a quasi-Gorenstein projective resolution of $M$, then $X_\bullet^{\oplus n}$ is a quasi-Gorenstein projective resolution of $M^{\oplus n}.$ Conversely, if $Y_\bullet$ is a quasi-Gorenstein projective resolution of $M^{\oplus n}$, then it is also a quasi-Gorenstein projective resolution of $M.$ The assertion now follows.

(2) We may assume $\operatorname{qGpd}_{\mathcal{A}} M<\infty$ and $\operatorname{qGpd}_{\mathcal{A}} N<\infty.$ Let $X_\bullet$ and $X^{\prime}_\bullet$ be finite quasi-Gorenstein projective resolutions of $M$ and $N$ such that $\operatorname{qGpd}_{\mathcal{A}}M=\sup X_\bullet-\operatorname{hsup} X_\bullet$ and $\operatorname{qGpd}_{\mathcal{A}}N=\sup X_\bullet^{\prime}-\operatorname {hsup} X_\bullet^{\prime}$. Then $\mathrm{H}_i(X_\bullet) \cong M^{\oplus a_i}$ and $\mathrm{H}_j\left(X_\bullet^{\prime}\right) \cong N^{\oplus b_j}$ for some $a_i, b_j \geqslant 0$ (and all but finitely many
of the $a_i$ and $b_j$ are zero). Moreover, for any \(G \in \mathcal{GP}(\mathcal{A})\) and every $i \in \mathbb{Z}$, there are two natural isomorphisms $ \operatorname{H}_i(\operatorname{Hom}_{\mathcal{A}}(G, X_{\bullet})) \cong \operatorname{Hom}_{\mathcal{A}}(G, \operatorname{H}_i(X_{\bullet}))$ and $ \operatorname{H}_i(\operatorname{Hom}_{\mathcal{A}}(G, X^{\prime}_{\bullet})) \cong \operatorname{Hom}_{\mathcal{A}}(G, \operatorname{H}_i(X^{\prime}_{\bullet})).$
Then the complex
$$F_\bullet=(\oplus_{j \in \mathbb{Z}} X_\bullet^{\oplus b_j}[j]) \oplus\left(\oplus_{i \in \mathbb{Z}} X_\bullet^{\prime \oplus a_i}[i]\right)$$
is a quasi-Gorenstein projective resolution for $M \oplus N$; in fact, for each $k$, we have $\operatorname{H}_k(F_\bullet) = (M \oplus N)^{\oplus \sum_{i+j=k} a_i b_j}.$ Moreover,
for any \(G \in \mathcal{GP}(\mathcal{A})\), we get the natural isomorphism
$\operatorname{H}_k(\operatorname{Hom}_{\mathcal{A}}(G, F_{\bullet}))	\cong \operatorname{Hom}_{\mathcal{A}}(G, \operatorname{H}_k(F_{\bullet}))$ form the following formulas  
\begin{gather*}
	\operatorname{H}_k(\operatorname{Hom}_{\mathcal{A}}(G, (\oplus_{j \in \mathbb{Z}} X_\bullet^{\oplus b_j}[j]) \oplus (\oplus_{i \in \mathbb{Z}} X_\bullet^{\prime \oplus a_i}[i]))) \\
	\cong (\oplus_{j \in \mathbb{Z}} \operatorname{H}_{k-j}(\operatorname{Hom}_{\mathcal{A}}(G, X_\bullet))^{\oplus b_j}) \oplus (\oplus_{i \in \mathbb{Z}} \operatorname{H}_{k-i}(\operatorname{Hom}_{\mathcal{A}}(G, X'_\bullet))^{\oplus a_i}) \\
	\cong (\oplus_{j \in \mathbb{Z}} \operatorname{Hom}_{\mathcal{A}}(G, \operatorname{H}_{k-j}(X_\bullet))^{\oplus b_j}) \oplus (\oplus_{i \in \mathbb{Z}} \operatorname{Hom}_{\mathcal{A}}(G, \operatorname{H}_{k-i}(X'_\bullet))^{\oplus a_i}) \\
	\cong \operatorname{Hom}_{\mathcal{A}}(G, (\oplus_{j \in \mathbb{Z}} \operatorname{H}_{k-j}(X_\bullet)^{\oplus b_j}) \oplus (\oplus_{i \in \mathbb{Z}} \operatorname{H}_{k-i}(X'_\bullet)^{\oplus a_i})).
\end{gather*}
 Since $\operatorname{sup} F_\bullet=\max \left\{\operatorname{sup} X_\bullet+\operatorname{hsup} X_\bullet^{\prime}, \operatorname{sup} X_\bullet^{\prime}+\operatorname{hsup} X_\bullet\right\}$ and $\operatorname{hsup}
 F_\bullet=\operatorname{hsup} X_\bullet+\operatorname{hsup} X_\bullet^{\prime}$, we have that
	$\operatorname{qGpd}_{\mathcal{A}}(M \oplus N) \leqslant \operatorname{sup} F_\bullet-\operatorname{hsup}F_\bullet=\max \left\{\operatorname{sup} X_\bullet
	-\operatorname{hsup} X_\bullet, \operatorname{sup} X_\bullet^{\prime}-\operatorname{hsup} X_\bullet^{\prime}\right\}= \operatorname{sup} \left\{\operatorname{qGpd}_{\mathcal{A}} M, \operatorname{qGpd}_{\mathcal{A}}N\right\}.$ 

(3) The assertion follows from (2) by letting $N=G$, in which case $\operatorname{qGpd}_\mathcal{A}G \leqslant \operatorname{Gpd}_\mathcal{A}G = 0$.

(4) Consider the proper Gorenstein projective resolution of $M$ $$G_\bullet :\cdots\xrightarrow{d_{n+1}} G_{n}\xrightarrow{d_n} G_{n-1}\xrightarrow{d_{n-1}} \cdots \xrightarrow{d_2} G_1 \xrightarrow{d_1} G_0\xrightarrow{d_0} M \rightarrow 0$$ with $\operatorname{Im}d_{r}\cong M$, for some $r\geqslant 1$. Let $s = \max\{1, r-1\}$. Then the truncated complex
$$G'_\bullet = (0 \to G_s \xrightarrow{d_s} \cdots \xrightarrow{d_1} G_0 \to 0)$$
is a complex of Gorenstein projective objects with $\operatorname{H}_0(G'_\bullet)\cong M$ and $\operatorname{H}_i(G'_\bullet) = 0$ for $0 < i < s$. If $r\geqslant 2$, then $\operatorname{H}_s(G'_\bullet)= \operatorname{Ker}d_s=\operatorname{Ker}d_{r-1}\cong\operatorname{Im}d_{r}\cong M$.
If $r=1$ then $\operatorname{Im}d_{1}\cong M$, and thus we may choose $G_1 =G_0$ 
as the proper Gorenstein projective cover of $\operatorname{Im}d_{1}$, which yields $\operatorname{H}_s(G'_\bullet) = \operatorname{Ker}d_{1}\cong \operatorname{Ker}d_{0}\cong M$.
Since $G_\bullet$ is a proper Gorenstein projective resolution, we have the natural isomorphism
$\operatorname{H}_i(\operatorname{Hom}_{\mathcal{A}}(G, G'_\bullet))	\cong \operatorname{Hom}_{\mathcal{A}}(G, \operatorname{H}_i(G'_\bullet))$ for any \(G \in \mathcal{GP}(\mathcal{A})\) and every $i \in \mathbb{Z}$.
Therefore, $G'_\bullet$ is a quasi-Gorenstein projective resolution of $M$.
Since $H_s(G'_\bullet)\neq0$, we get $\operatorname{qGpd}_\mathcal{A} M \leqslant \sup(G'_\bullet) -\operatorname{hsup}(G'_\bullet) =s-s= 0$.
\end{proof}

As proved in \cite[Proposition~3.3 (4)]{GJT21}, the quasi-projective dimension of an object is at least that of its first syzygy. Motivated by this, we consider the following proposition, the proof of which shows that condition (iii) in Definition~\ref*{quasi} (1) is truly necessary.

\begin{proposition}\label{short exact sequence 1}
Let $0 \rightarrow N \rightarrow G \xrightarrow{\pi} M \rightarrow 0$ 
be a proper exact sequence in $\mathcal{A}$ such that $G$ is Gorenstein projective. Then
$\operatorname{qGpd}_{\mathcal{A}}N\leqslant \operatorname{qGpd}_{\mathcal{A}}M$.
\end{proposition}	
     \begin{proof} 
     	The inequality automatically holds if $\operatorname{qGpd}_{\mathcal{A}}M=\infty.$ Now we assume  $\operatorname{qGpd}_{\mathcal{A}} M<\infty$. Let $X_\bullet$ be a finite quasi-Gorenstein projective resolution of $M$ such that $\operatorname{qGpd}_{\mathcal{A}} M=\sup X_\bullet-\operatorname{hsup} X_\bullet.$ Then for each $i \in \mathbb{Z},$ there is an isomorphsim $g_i\colon M^{\oplus a_i}\rightarrow \operatorname{H}_i(X_\bullet)$ for some integers $a_i\in \mathbb{N} $ (not all $a_i$ are zero), and for any $G' \in \mathcal{GP(A)}$ the morphism
     	$$l_i^{X_\bullet}\colon \operatorname{H}_i(\operatorname{Hom}_\mathcal{A}(G', X_\bullet)) \to \operatorname{Hom}_\mathcal{A}(G', \operatorname{H}_i(X_\bullet))$$ induced by
     	$\operatorname{Hom}_\mathcal{A}(G', -)$
     	is an isomorphism. 
     	Set $Q_i=G^{\oplus a_i}$ and consider the complex
      	\[Q_\bullet=\left(\cdots \rightarrow Q_{i+1} \xrightarrow{0} Q_i \xrightarrow{0} Q_{i-1} \rightarrow \cdots\right).\] 
      	Applying $\operatorname{Hom}_{\mathcal{A}}(Q_i, -)$ to the canonical epimorphism
      	$\pi_i \colon \operatorname{Z}_i(X_{\bullet}) \to \operatorname{H}_i(X_{\bullet})$, we get an epimorphism
      	$\operatorname{Hom}_{\mathcal{A}}(Q_i, \pi_i)\colon \operatorname{Hom}_{\mathcal{A}}(Q_i, \operatorname{Z}_i(X_\bullet)) \to \operatorname{Hom}_{\mathcal{A}}(Q_i, \operatorname{H}_i(X_\bullet))$ by Lemma \ref{naturally}$.$ Therefore, the morphism $g_i\pi^{\oplus a_i}\colon Q_i \rightarrow \operatorname{H}_i(X_\bullet)$ lifts to a morphism $\beta_i\colon Q_i \rightarrow \mathrm{Z}_i(X_\bullet)$ such that $\pi_i\beta_i=g_i\pi^{\oplus a_i}.$ Composing this morphism $\beta_i$ with the inclusion $\eta_i\colon \mathrm{Z}_i(X_\bullet) \rightarrow X_i,$ we get a morphism $\alpha_i=\eta_i\beta_i\colon Q_i \rightarrow X_i.$ Let $d_\bullet$ be the differential of $X_\bullet.$
      	 Then we have $d_i\alpha_i=d_i(\eta_i\beta_i)=(d_i\eta_i)\beta_i=0$, which gives rise to a chain map $\alpha_\bullet\colon Q_\bullet \rightarrow X_\bullet$. The short exact sequence of complexes $0 \rightarrow X_\bullet \rightarrow \operatorname{Cone}(\alpha_\bullet) \rightarrow Q_\bullet[1] \rightarrow 0$ yields a long exact sequence of homology
      	$$\cdots \xrightarrow{\operatorname{H}_{i+1}(\alpha_{\bullet})} \operatorname{H}_{i+1}(X_{\bullet}) \rightarrow \operatorname{H}_{i+1}(\operatorname{Cone}(\alpha_{\bullet})) \rightarrow \operatorname{H}_i(Q_{\bullet}) \xrightarrow{\operatorname{H}_i(\alpha_{\bullet})} \operatorname{H}_i(X_{\bullet})\rightarrow\cdots .$$
     	As $\operatorname{H}_i\left(\alpha_{\bullet}\right)=g_i\pi^{\oplus a_i}$ is an epimorphism, the above long exact sequence breaks into short exact sequences
     	\begin{equation}
     		0 \rightarrow \operatorname{H}_{i+1}(\operatorname{Cone}(\alpha_\bullet)) \rightarrow \operatorname{H}_i(Q_\bullet) \xrightarrow{g_i\pi^{\oplus a_i}} \operatorname{H}_i(X_\bullet) \rightarrow 0, \label{eq:short_exact_seq}
     	\end{equation}
     	and then $\operatorname{H}_{i+1}(\operatorname{Cone}(\alpha_\bullet))\cong\operatorname{Ker}(g_i\pi^{\oplus a_i})\cong\operatorname{Ker}(\pi^{\oplus a_i})\cong N^{\oplus a_i}$ for all $i \in \mathbb{Z}$. Therefore, to show $\operatorname{Cone}(\alpha_\bullet)$  is a quasi-Gorenstein projective resolution of $N,$ it remains to prove that the morphism
     	$$l_i ^{\operatorname{Cone}(\alpha_\bullet)}\colon \operatorname{H}_i(\operatorname{Hom}_\mathcal{A}(G', \operatorname{Cone}(\alpha_\bullet))) \to \operatorname{Hom}_\mathcal{A}(G', \operatorname{H}_i(\operatorname{Cone}(\alpha_\bullet)))$$ is an isomorphsim for any $i \in \mathbb{Z}$ and any $G' \in \mathcal{GP(A)}$. 
     	
     	Applying $\operatorname{Hom}_\mathcal{A}(G', -)$ to the  
     	split exact sequence of complexes $0\to X_\bullet \to \Cone(\alpha_\bullet) \to Q_\bullet[1]\to 0$, we get the following short exact sequence  $$0 \to \operatorname{Hom}_\mathcal{A}(G', X_\bullet) \to \operatorname{Hom}_\mathcal{A}(G', \operatorname{Cone}(\alpha_\bullet)) \to \operatorname{Hom}_\mathcal{A}(G', Q_\bullet[1]) \to 0,$$     	
  which yields a long exact sequence   
     \begin{gather}
     	\cdots \xrightarrow{\operatorname{H}_{i+1}((\alpha_\bullet)\ast)} \operatorname{H}_{i+1}(\operatorname{Hom}_\mathcal{A}(G', X_\bullet)) \longrightarrow \operatorname{H}_{i+1}(\operatorname{Hom}_\mathcal{A}(G', \operatorname{Cone}(\alpha_\bullet))) \nonumber\\
     	\longrightarrow \operatorname{H}_i(\operatorname{Hom}_\mathcal{A}(G', Q_\bullet))
     	\xrightarrow{\operatorname{H}_i((\alpha_\bullet)\ast)} \operatorname{H}_i(\operatorname{Hom}_\mathcal{A}(G', X_\bullet)) \longrightarrow \cdots.
     	\label{long-exact-ast}
     \end{gather}
   By Lemma \ref{commutate}, we have the following exact commutative diagram
   \[
   \begin{tikzcd}[column sep=24pt]
   	\operatorname{H}_{i+1}(\operatorname{Hom}_\mathcal{A}(G', \operatorname{Cone}(\alpha_\bullet))) 
   	\arrow[r] 
   	\arrow[d, "l_{i+1}^{\operatorname{Cone}(\alpha_\bullet)}" ]
   	& \operatorname{H}_i(\operatorname{Hom}_\mathcal{A}(G', Q_\bullet)) 
   	\arrow[r, "\operatorname{H}_i((\alpha_\bullet)_\ast)"] 
   	\arrow[d, "l_i^{Q_\bullet}" ]
   	& \operatorname{H}_i(\operatorname{Hom}_\mathcal{A}(G', X_\bullet)) 
   	\arrow[d, "l_i^{X_\bullet}" ] \\
   	\operatorname{Hom}_\mathcal{A}(G', \operatorname{H}_{i+1}(\operatorname{Cone}(\alpha_\bullet))) 
   	\arrow[r, hook] 
   	& \operatorname{Hom}_\mathcal{A}(G', \operatorname{H}_i(Q_\bullet)) 
   	\arrow[r, "(g_i\pi^{\oplus a_i})_\ast"] 
   	& \operatorname{Hom}_\mathcal{A}(G', \operatorname{H}_i(X_\bullet)),
   \end{tikzcd}
   \]
   where the second row is obtained by applying $\operatorname{Hom}_\mathcal{A}(G', -)$ to ~(\ref{eq:short_exact_seq}).	Since the differentials of $Q_\bullet$ are zero,
   it follows that $l_i^{Q_\bullet}$ is an isomorphism. On the other hand, the proper exactness of \(0 \to N \to G \xrightarrow{\pi} M \to 0\) implies that  $\pi_\ast=\operatorname{Hom}_\mathcal{A}(G', \pi)$ is an epimorphism, and hence  $(g_i\pi^{\oplus a_i})_\ast$ is an epimorphism as well.
     Since $l_i^{X_\bullet}$ is an isomorphism, it follows from the above commutative diagram that $\operatorname{H}_i((\alpha_\bullet)_\ast)$is an epimorphism$.$ Then the long exact sequence \eqref{long-exact-ast} breaks into short exact sequences
	$0 \to \operatorname{H}_{i+1}(\operatorname{Hom}_\mathcal{A}(G', \operatorname{Cone}(\alpha_\bullet))) \to \operatorname{H}_i(\operatorname{Hom}_\mathcal{A}(G', Q_\bullet)) \xrightarrow{\operatorname{H}_i((\alpha_\bullet)\ast)} \operatorname{H}_i(\operatorname{Hom}_\mathcal{A}(G', X_\bullet)) \to 0.$ Applying the Snake Lemma to the above commutative diagram, we obtain that
	$l_{i+1}^{\operatorname{Cone}(\alpha_\bullet)}$ is an isomorphism for any  $i \in \mathbb{Z}$. 
	Therefore, \begin{center}
		$\operatorname{Cone}(\alpha_\bullet)=\left(\cdots \rightarrow Q_i \oplus X_{i+1} \rightarrow Q_{i-1} \oplus X_i \rightarrow \cdots\right)$
	\end{center}
     is a quasi-Gorenstein projective resolution of $N.$
     Since $\operatorname{H}_{i+1}(\operatorname{Cone}(\alpha_\bullet))\cong N^{\oplus a_i}$, we have $\operatorname{hsup(Cone}(\alpha_\bullet)) =\operatorname{hsup}X_\bullet+1.$ Moreover, it is clear that $\operatorname{sup(Cone}(\alpha_\bullet))\leqslant \sup X_\bullet+1$. Therefore, we have  
      	$\operatorname{qGpd}_{\mathcal{A}}N \leqslant \sup (\operatorname{Cone}(\alpha_\bullet))-\operatorname{hsup} (\operatorname{Cone}(\alpha_\bullet)) \leqslant (\operatorname{sup} X_\bullet+1)-(\operatorname{hsup} X_\bullet+1)=\sup X_\bullet-\operatorname{hsup}X_\bullet=\operatorname{qGpd}_{\mathcal{A}} M$.   
     \end{proof}
     
 Motivated by \cite[Proposition~3.2]{DFG}, we consider the following proposition.     
\begin{proposition}\label{short exact sequence 2}
	Let $0 \to G \xrightarrow{f} M \xrightarrow{g} N \to 0 $ be a proper exact sequence in $\mathcal{A}$ such that $G$ is Gorenstein projective. Then $\mathrm{qGpd}_\mathcal{A}N \leqslant \sup\{1, \mathrm{qGpd}_\mathcal{A}M\}$.
\end{proposition}
\begin{proof}
	We may assume that $M$ has finite quasi-Gorenstein projective dimension. Let $X_\bullet$ be a bounded quasi-Gorenstein projective resolution of $M$ such that $\mathrm{qGpd}_\mathcal{A}M = \sup X_\bullet -\operatorname{hsup} X_\bullet.$ Then for each $i \in \mathbb{Z},$ there is an isomorphsim $h_i\colon M^{\oplus a_i}\rightarrow \operatorname{H}_i(X_\bullet)$ for some integers $a_i\in \mathbb{N} $ (not all $a_i$ are zero), and for any $G' \in \mathcal{GP(A)}$ the morphism
	$l_i^{ X_\bullet} \colon \operatorname{H}_i(\operatorname{Hom}_\mathcal{A}(G', X_\bullet)) \to \operatorname{Hom}_\mathcal{A}(G', \operatorname{H}_i(X_\bullet))$ induced by $\operatorname{Hom}_\mathcal{A}(G',-)$ 
	is an isomorphism$.$ Let $Q_i = G^{\oplus a_i}$ and consider the complex
	\[Q_\bullet=(\cdots \rightarrow Q_{i+1} \xrightarrow{0} Q_i \xrightarrow{0} Q_{i-1} \rightarrow \cdots).\]
	By a similar argument as in Proposition~\ref{short exact sequence 1}, we have a chain map $\alpha_\bullet\colon Q_\bullet \rightarrow X_\bullet$ with $\operatorname{H}_i\left(\alpha_{\bullet}\right)=h_if^{\oplus a_i}.$
	Also, the short exact sequence of complexes $0 \rightarrow X_\bullet \rightarrow \operatorname{Cone}(\alpha_\bullet) \rightarrow Q_\bullet[1] \rightarrow 0$ yields the long exact sequence of homology $$\cdots \rightarrow \operatorname{H}_i\left(Q_{\bullet}\right) \xrightarrow{\operatorname{H}_i(\alpha_{\bullet})} \operatorname{H}_i(X_\bullet)) \xrightarrow{\varphi_i} \operatorname{H}_i(\operatorname{Cone}\left(\alpha_{\bullet}\right)) \rightarrow \operatorname{H}_{i-1}(Q_\bullet)\xrightarrow{\operatorname{H}_{i-1}(\alpha_\bullet)}\cdots,$$
	which breaks into short exact sequences
	\begin{equation}
		0\rightarrow \operatorname{H}_i(Q_\bullet) \xrightarrow{h_if^{\oplus a_i}} \mathrm{H}_i(X_\bullet) \xrightarrow{\varphi_i} \operatorname{H}_i(\operatorname{Cone}(\alpha_\bullet)) \rightarrow 0, \label{eq:short_exact_seq-1}
	\end{equation}
 as $\operatorname{H}_i\left(\alpha_{\bullet}\right)=h_if^{\oplus a_i}$ is monomorphic.
Then for all $i \in \mathbb{Z},$ we have $\operatorname{H}_i(\operatorname{Cone}(\alpha_\bullet))
\\ \cong
\operatorname{Coker}(h_if^{\oplus a_i})\cong\operatorname{Coker}(f^{\oplus a_i})\cong N^{\oplus a_i}$ and $\varphi_i=g^{\oplus a_i}h_i^{-1}$. 
Now we claim that the induced morphism $$l_{i}^{\operatorname{Cone}(\alpha_\bullet)}\colon \operatorname{H}_i(\operatorname{Hom}_\mathcal{A}(G', \operatorname{Cone}(\alpha_\bullet))) \to \operatorname{Hom}_\mathcal{A}(G', \operatorname{H}_i(\operatorname{Cone}(\alpha_\bullet)))$$ is an isomorphism, for any $G' \in\mathcal{GP(A)}$. Indeed, this can be proved in a manner similar to that in Proposition~\ref{short exact sequence 1}.
Therefore, the complex
		\begin{center}
			$\operatorname{Cone}(\alpha_\bullet)=\left(\cdots \rightarrow Q_i \oplus X_{i+1} \rightarrow Q_{i-1} \oplus X_i \rightarrow \cdots\right)$
		\end{center}
		is a quasi-Gorenstein projective resolution of $N.$ Since $\operatorname{H}_i(\operatorname{Cone}(\alpha_\bullet))\cong N^{\oplus a_i}$, we have  $\operatorname {hsup}(\operatorname{Cone}(\alpha_\bullet)) =\operatorname{hsup} X_\bullet.$ On the other hand, it follows from the definition of mapping cone and from the definition of the complex $Q_\bullet$ that 
		 $\sup(\operatorname{Cone}(\alpha_\bullet)) = \sup\{\operatorname{hsup}X_\bullet + 1, \sup X_\bullet\}.$ Therefore, we have \begin{center}
			$\operatorname{qGpd}_{\mathcal{A}}N \leqslant \operatorname{sup} (\operatorname{Cone}(\alpha_\bullet))-\operatorname{hsup} (\operatorname{Cone}(\alpha_\bullet)) \leqslant \operatorname{sup}\{1, \operatorname{sup} X_\bullet - \operatorname{hsup}X_\bullet\}$.
		\end{center}
		This shows that $\operatorname{qGpd}_\mathcal{A}N\leqslant\{1,\operatorname{qGpd}_{\mathcal{A}}M\}$.
	\end{proof}
 
 		The following proposition is inspired by \cite[Proposition~3.8]{CCL26}.
     
\begin{proposition} \label{short exact sequence 3}
	Let $0 \rightarrow N \xrightarrow{f} E \rightarrow M \rightarrow 0$ be a proper exact sequence in $\mathcal{A}$ such that $E$ is injective and Gorenstein projective$.$ Then $\operatorname{qGpd}_{\mathcal{A}}M \leqslant \operatorname{qGpd}_{\mathcal{A}}N+1$.
\end{proposition}
\begin{proof}
	This inequality automatically holds if $\operatorname{qGpd}_{\mathcal{A}}N=\infty.$ So, we assume $\operatorname{qGpd}_{\mathcal{A}}N=m<\infty.$ Let $X_\bullet$ be a finite quasi-Gorenstein projective resolution of $N$
	such that $\operatorname{qGpd}_{\mathcal{A}} N=\sup X_\bullet-\operatorname{hsup} X_\bullet.$ Then for each $i \in \mathbb{Z}$, there is an isomorphism $g_i\colon \operatorname{H}_i\left(X_{\bullet}\right) \rightarrow N^{\oplus a_i}$ for some integers $a_i \in \mathbb{N}$ (not all $a_i$ are zero), and for any $G \in \mathcal{GP(A)}$ the induced morphism
	$l_i^{X_\bullet} \colon \operatorname{H}_i(\operatorname{Hom}_\mathcal{A}(G, X_\bullet)) \to \operatorname{Hom}_\mathcal{A}(G, \operatorname{H}_i(X_\bullet))$ is an isomorphism$.$ 
		Let $E_i=E^{\oplus a_i}$ and consider the complex
	\begin{center}
		$E_\bullet=(\cdots \rightarrow E_{i+1} \xrightarrow{0} E_i \xrightarrow{0} E_{i-1} \rightarrow \cdots).$
	\end{center}
	Following the approach of \cite[Proposition~3.8]{CCL26} (where the injectivety of $E$ is used), we get a chain map $h_{\bullet}: X_{\bullet} \rightarrow E_{\bullet}$
	such that $\operatorname{H}_i(h_{\bullet})= f^{\oplus a_i}g_i$ and $\operatorname{H}_i(\operatorname{Cone}(h_\bullet))\cong M^{\oplus a_i}$. Following the proof of
	 Proposition~\ref{short exact sequence 1}, we can show that the induced morphism
	 $$l_i^{\operatorname{Cone}(h_\bullet)} \colon \operatorname{H}_i(\operatorname{Hom}_\mathcal{A}(G, \operatorname{Cone}(h_\bullet))) \to \operatorname{Hom}_\mathcal{A}(G, \operatorname{H}_i(\operatorname{Cone}(h_\bullet)))$$ is an isomorphism, for any $i \in \mathbb{Z}$ and any $G \in \mathcal{GP(A)}$.
	Therefore, \begin{center}
		$\operatorname{Cone}(h_\bullet)=\left(\cdots \rightarrow X_i \oplus E_{i+1} \rightarrow X_{i-1} \oplus E_i \rightarrow \cdots\right)$
	\end{center}is a quasi-Gorenstein projective resolution of $M.$ Since $\operatorname{H}_i(\operatorname{Cone}(h_\bullet))\cong M^{\oplus a_i}$, we get $\operatorname {hsup(Cone}(h_\bullet)) =\operatorname{hsup}X_\bullet$. Moreover, it is clear that
	$\sup(\operatorname{Cone}(h_\bullet))=\sup X_\bullet+1$. Therefore, we have 
		$\operatorname{qGpd}_{\mathcal{A}}M \leqslant \sup (\operatorname{Cone}(h_\bullet))-\operatorname{hsup} (\operatorname{Cone}(h_\bullet)) \leqslant \operatorname{sup} X_\bullet+1-\operatorname{hsup} X_\bullet=\operatorname{qGpd}_{\mathcal{A}} N+1$.
\end{proof} 

Let $M$ be an object of $\mathcal{A}$. Then it follows from the definition of quasi-Gorenstein projective dimension that $\operatorname{qGpd}_{\mathcal{A}}M\leqslant\operatorname{Gpd}_{\mathcal{A}}M$.
Motivated by {\rm\cite[Proposition 3.1]{CCL26}}, we will show that this inequality becomes an equality when $\operatorname{Gpd}_{\mathcal{A}}M < \infty$.
        
\begin{proposition}\label{dim}
	Let $M$ be an object of $\mathcal{A}$ with $\operatorname{Gpd}_{\mathcal{A}}M < \infty.$ Then $\operatorname{qGpd}_{\mathcal{A}}M= \operatorname{Gpd}_{\mathcal{A}}M$. 
\end{proposition}
\begin{proof}
Since $\operatorname{qGpd}_{\mathcal{A}}M \leqslant \operatorname{Gpd}_{\mathcal{A}}M < \infty$, we may assume that $r=\operatorname{qGpd}_{\mathcal{A}}M$. By shifting, we obtain a quasi-Gorenstein projective resolution of $M$ of the form
		\begin{center}
			$X_{\bullet}:\quad 0 \rightarrow X_r \xrightarrow{d_r} X_{r-1} \xrightarrow{d_{r-1}} \cdots\xrightarrow{d_2} X_1 \xrightarrow{d_1} X_0 \xrightarrow{d_0} X_{-1}\cdots \xrightarrow{d_{-s+1}} X_{-s} \to 0,$
		\end{center}
	where $X_r \neq 0, \  \operatorname{hsup} X_{\bullet}=0$ and $r=\operatorname{Gpd}_{\mathcal{A}}(\operatorname {Coker}(d_1))$. Let $N= \operatorname {Coker}(d_1)$. It then follows that $\operatorname{qGpd}_{\mathcal{A}} M =r= \operatorname{Gpd}_{\mathcal{A}} N$. Now we claim $\operatorname{Gpd}_{\mathcal{A}} M =\operatorname{Gpd}_{\mathcal{A}} N$.
	
Since $ \operatorname{hsup} X_{\bullet}=0$, it follows that $\operatorname{H}_0(X_\bullet)\neq 0$, which implies $X_0\neq 0$. 
If $s = 0$, then $N \cong \operatorname{H}_0(X_\bullet) \cong M^{\oplus a_0}$, yielding $\operatorname{Gpd}_{\mathcal{A}} M = \operatorname{Gpd}_{\mathcal{A}} N.$ 
Now assume that $s > 0.$ Following the proof of {\rm\cite[Proposition 2.2 ]{CCL26}}, we obtain the following short exact sequences:
\begin{equation}\label{exact-1}
	0 \to M^{\oplus a_0} \to N \to \mathrm{Im}d_0 \to 0, 
\end{equation}
\begin{equation}\label{exact-2}
	0 \to \mathrm{Imd}_{i} \to \mathrm{Ker} d_{i-1} \to M^{\oplus a_{i-1}} \to 0, 
\end{equation}
\begin{equation}\label{exact-3}
	0 \to \mathrm{Ker} d_{i-1} \to P_{i-1} \to \mathrm{Im}d_{i-1} \to 0.
\end{equation}
Assume that $ \operatorname{Gpd}_{\mathcal{A}}M=n.$ Then it follows from {\rm\cite[Theorem 2.20]{Hol}}
that \(\operatorname{Ext}_\mathcal{A}^i(M, P) = 0\) for all $i > n$ and all $P \in \mathcal{P(A)}$, and there exists an object $P' \in \mathcal{P(A)}$ such that $\operatorname{Ext}_\mathcal{A}^n(M, P') \neq 0.$ Applying $\operatorname{Hom}_\mathcal{A}(-, P)$ to ~(\ref{exact-1}), we have a long exact sequence
\begin{equation} \label{long_exact-1}
	\begin{aligned}
		\cdots \to \operatorname{Ext}_\mathcal{A}^{n+j}(\operatorname{Im} d_0, P) 
		&\to \operatorname{Ext}_\mathcal{A}^{n+j}(N, P) 
		\to \operatorname{Ext}_\mathcal{A}^{n+j}(M^{\oplus a_0}, P)& \\
		&\to \operatorname{Ext}_\mathcal{A}^{n+j+1}(\operatorname{Im} d_0, P) 
		\to \cdots &.
	\end{aligned}
\end{equation}
	Consequently, if $\operatorname{Ext}_\mathcal{A}^{n+j}(\operatorname{Im} d_0, P) = 0$ for all $j > 0$ and all $P \in \mathcal{P(A)}$, 
	then 
    \begin{center}
    	$0 = (\operatorname{Ext}_\mathcal{A}^{n+j}(M, P))^{\oplus a_0} \cong \operatorname{Ext}_\mathcal{A}^{n+j}(M^{\oplus a_0}, P) \cong \operatorname{Ext}_\mathcal{A}^{n+j}(N, P),$
    \end{center}
  and the morphism $\operatorname{Ext}_\mathcal{A}^n(N, P') \rightarrow \operatorname{Ext}_\mathcal{A}^n(M^{\oplus a_0}, P')$ in~(\ref{long_exact-1}) 
   is an epimorphism$.$ Since $\operatorname{Ext}_\mathcal{A}^n(M, P') \neq 0$, we deduce that $\operatorname{Ext}_\mathcal{A}^n(N, P') \neq 0$. Hence, $\operatorname{Gpd}_{\mathcal{A}} N=n= \operatorname{Gpd}_{\mathcal{A}} M$, as desired.
	
    It remains to show $\operatorname{Ext}_\mathcal{A}^{n+j}(\operatorname{Im} d_0, P) = 0$ for all $j>0$ and all $P \in \mathcal{P(A)}$. 
     Applying \(\operatorname{Hom}_\mathcal{A}(-, P)\) to~(\ref{exact-2}), and using the fact that $\operatorname{Ext}_\mathcal{A}^{i}(M, P)=0$ for any $i>n$, we have
     $ \operatorname{Ext}_\mathcal{A}^{n+j}(\operatorname{Ker}d_{-1}, P) \cong \operatorname{Ext}_\mathcal{A}^{n+j}(\operatorname{Im}d_0, P).$
     Similarly, applying $\operatorname{Hom}_\mathcal{A}(-, P)$ to exact sequence~(\ref{exact-3}), we obtain 
      $\operatorname{Ext}_\mathcal{A}^{n+j}(\operatorname{Ker}d_{-1}, P) \cong \operatorname{Ext}_\mathcal{A}^{n+j+1}(\operatorname{Im}d_{-1}, P).$ 
      Consequently,
      we have a series of isomorphisms:
     \begin{center}
     	$\operatorname{Ext}_\mathcal{A}^{n+j}(\operatorname{Im}d_0, P) \cong \operatorname{Ext}_\mathcal{A}^{n+j+1}(\operatorname{Im}d_{-1}, P) \cong \cdots \cong \operatorname{Ext}_\mathcal{A}^{n+j+s}(\operatorname{Im}d_{-s}, P) =0,$
     \end{center}
      where the last equality follows from the vanishing of the morphism \(d_{-s}\colon P_{-s} \to 0\).
\end{proof}

For a complex $X_{\bullet}=(X_i, d_i^X)_{i\in \mathbb{Z}}$, we denote the good truncations of $X_\bullet$ by
$\tau_{\geqslant n}X_\bullet$ and $\tau_{\leqslant n}X_\bullet$, that is,
$$\tau_{\geqslant n}X_\bullet=\  \cdots\xrightarrow{d_{n+3}}X_{n+2}\xrightarrow{d_{n+2}}X_{n+1}\rightarrow\operatorname{Ker}d_n\rightarrow 0\rightarrow 0\rightarrow \cdots,$$
$$\tau_{\leqslant n}X_\bullet=\ \cdots\rightarrow 0 \rightarrow 0 \rightarrow \operatorname{Im}d_{n+1} \rightarrow X_n\xrightarrow{d_n} X_{n-1}\xrightarrow{d_{n-1}} \rightarrow \cdots.$$
The following lemma is inspired by \cite[Lemma~3.2]{CCL26}.
Here,  $\operatorname{Ext}^n_{\mathcal{G}}(M, N) = 0$ is the relative cohomology groups defined in \cite{AM02} via a proper Gorenstein projective resolution of $M$.

\begin{lemma}\label{ext}
	Let $M$ be an object of $\mathcal{A}$ admitting a proper Gorenstein projective resolution, and let $k \in \mathbb{N}.$
	Suppose that 
	 $\operatorname{Ext}^n_{\mathcal{G}}(M, M) = 0$ for $2 \leqslant n \leqslant k + 1$, and that $X_\bullet \in \mathcal{C}^{b}(\mathcal{A})$ is such that  $\operatorname{H}_i(X_\bullet) \in \operatorname{add}(M)$ for  any $0 \leqslant i \leqslant k$, and $\operatorname{H}_i(X_\bullet) = 0$ for any $i > k$ or $i < 0.$ If the induced morphism 
	$l_i^{X_\bullet} \colon \operatorname{H}_i(\operatorname{Hom}_\mathcal{A}(G, X_\bullet)) \to \operatorname{Hom}_\mathcal{A}(G, \operatorname{H}_i(X_\bullet))$ is an isomorphism
	for any \(i \in \mathbb{Z}\) and any $G \in \mathcal{GP(A)}$, then $X_\bullet \cong \bigoplus_{i=0}^k \operatorname{H}_i(X_\bullet)[i]$ in $\mathcal D_{\mathrm{gp}}^b(\mathcal{A})$.
\end{lemma}
\begin{proof}
	We prove this statement by induction on $k.$ If $k=0$ then the statement holds trivially$.$ Let $k > 0.$ For any $G \in \mathcal{GP(A)}$, it follows from Lemma~\ref{naturally} that $\operatorname{Hom}_\mathcal{A}(G, \alpha _i)\colon \operatorname{Hom}_\mathcal{A}(G, X_k)\to \operatorname{Hom}_\mathcal{A}(G, \operatorname{Im}d_k)$ is an epimorphism, where
	$\alpha _i\colon X_k \to \operatorname{Im}d_k$ is the canonical epimorphism.
	Therefore, the
	short exact sequence of complexes \begin{equation}\label{exact-tri-3}
		0 \to \tau_{\geqslant k}X_\bullet \to X_\bullet \to \tau_{\leqslant k-1}X_\bullet \to 0
	\end{equation}
	remains exact under the functor $\operatorname{Hom}_\mathcal{A}(G, -)$. Then
	we have a distinguished triangle in $\mathcal D_{\mathrm{gp}}^b(\mathcal{A})$:
	\begin{equation} \label{eq:truncation_triangle}
		\tau_{\geqslant k}X_\bullet \to X_\bullet \to \tau_{\leqslant k-1}X_\bullet \xrightarrow{f} (\tau_{\geqslant k}X_\bullet)[1],
	\end{equation}
	where $\tau_{\geqslant k}X_\bullet \cong \operatorname{H}_k(X_\bullet)[k]\in\operatorname{add}(M[k])$ and $\operatorname{H}_i(\tau_{\leqslant k-1}X_\bullet) \cong \operatorname{H}_i(X_\bullet)$ for $i \leqslant k-1$ but $\operatorname{H}_i(\tau_{\leqslant k-1}X_\bullet) = 0$ for $i \geqslant k.$ 
	In particular, $\operatorname{H}_i(\tau_{\leqslant {k-1}}X_\bullet) \in \operatorname{add}(M)$ for $0 \leqslant i \leqslant {k-1}.$ Now we claim that the induced morphism $$l_i^{\tau_{\leqslant {k-1}}X_\bullet} \colon \operatorname{H}_i(\operatorname{Hom}_\mathcal{A}(G, \tau_{\leqslant {k-1}}X_\bullet)) \to \operatorname{Hom}_\mathcal{A}(G, \operatorname{H}_i(\tau_{\leqslant {k-1}}X_\bullet))$$ is an isomorphism
	for any \(i \in \mathbb{Z}\) and any $G \in \mathcal{GP(A)}$, and then the induction can be applied to $\tau_{\leqslant {k-1}}X_\bullet$.
	
	Applying the functor $\operatorname{Hom}_\mathcal{A}(G, -)$ to (\ref{exact-tri-3}) yields a short exact sequence of complexes. Since the $i$-th term of $\tau_{\geqslant k}X_\bullet$ is zero for any $i \leqslant k-1$,
	 the induced long exact sequence in homology gives an isomorphism $$\operatorname{H}_i(\operatorname{Hom}_\mathcal{A}(G, X_\bullet)) \cong  \operatorname{H}_i(\operatorname{Hom}_\mathcal{A}(G, \tau_{\leqslant k-1} X_\bullet)).$$ Then for any $i \leqslant k-1$, 
	 we have the following commutative diagram
	\[
	\begin{tikzcd}
		\operatorname{H}_i(\operatorname{Hom}_\mathcal{A}(G, X_\bullet)) 
		\arrow[r, "\cong"] 
		\arrow[d, "l_i^{X_\bullet}"]
		& \operatorname{H}_i(\operatorname{Hom}_\mathcal{A}(G, \tau_{\leqslant k-1} X_\bullet)) 
		\arrow[d, "l_i^{\tau_{\leqslant k-1}X_\bullet}" ] \\
		\operatorname{Hom}_\mathcal{A}(G, \operatorname{H}_i(X_\bullet)) 
		\arrow[r, "\cong"]
		& \operatorname{Hom}_\mathcal{A}(G, \operatorname{H}_i(\tau_{\leqslant k-1}X_\bullet))
	\end{tikzcd}
	\]
	by Lemma~\ref{naturally}, where the second row is an isomorphism since $\operatorname{H}_i(X_\bullet) \cong \operatorname{H}_i(\tau_{\leqslant k-1}X_\bullet)  $. 
	As $l_i^{X_\bullet}$ is an isomorphism, we deduce that 
	 $l_i^{\tau_{\leqslant {k-1}}X_\bullet}$ is an isomorphism 
	for any $i \leqslant k-1$. For any $i > k$, 
	the morphism $l_i^{\tau_{\leqslant {k-1}}X_\bullet}$ is an isomorphism as the $i$-th term of $\tau_{\leqslant {k-1}}X_\bullet$ vanishes.
	For $i=k$, both $H_k(\tau_{\leqslant {k-1}}X_\bullet)$ and $\operatorname{H}_k(\operatorname{Hom}_\mathcal{A}(G, \tau_{\leqslant k-1} X_\bullet))$ are zero, because the inclusion $\operatorname{Im} d_{k} \hookrightarrow X_{k-1}$ is still a monomorphism upon application of $\operatorname{Hom}_\mathcal{A}(G,-)$. Therefore, $l_i^{\tau_{\leqslant {k-1}}X_\bullet} $ is an isomorphism
	for any \(i \in \mathbb{Z}\) and any $G \in \mathcal{GP(A)}$.
	
	Now, induction on $k-1$ implies that
	\begin{center}
		$\tau_{\leqslant k-1}X_\bullet \cong \bigoplus_{i=0}^{k-1} \operatorname{H}_i(\tau_{\leqslant k-1}X_\bullet)[i] \cong \bigoplus_{i=0}^{k-1} \operatorname{H}_i(X_\bullet)[i] \in \operatorname{add}\left(\bigoplus_{i=0}^{k-1} M[i]\right)$.
	\end{center}
	Since $\operatorname{Ext}^n_{\mathcal{G}}(M, M) = 0$ for $2 \leqslant n \leqslant k+1,$ it follows from \cite[Theorem 3.12]{GZ10} that $\operatorname{Hom}_{\mathcal D_{\mathrm{gp}}^b(\mathcal{A})}
	({\textstyle\bigoplus_{i=0}^{k-1}} M[i], M[k+1])=0$, which implies that
	$\operatorname{Hom}_{\mathcal D_{\mathrm{gp}}^b(\mathcal{A})}
	(\tau_{\leqslant k-1}X_\bullet, (\tau_{\geqslant k}X_\bullet)[1]) 
	= 0.$
Thus, $f=0$ and the triangle~(\ref{eq:truncation_triangle}) splits, which yields the isomorphism $X_\bullet \cong \tau_{\geqslant k}X_\bullet \bigoplus \tau_{\leqslant k-1}X_\bullet \cong \bigoplus_{i=0}^k \operatorname{H}_i(X_\bullet)[i]$ 
in $\mathcal D_{\mathrm{gp}}^b(\mathcal{A})$. 
\end{proof}

Motivated by \cite[Theorem~3.3]{CCL26}, we consider the following proposition.
\begin{proposition}\label{relation}
	Suppose that $M \in \mathcal{A}$ has a finite quasi-Gorenstein projective resolution $X_\bullet$. If $M$ has a proper Gorenstein projective resolution, and  $\operatorname{Ext}^n_{\mathcal{G}}(M, M) = 0$ for $2 \leqslant n \leqslant \operatorname{hsup}X_\bullet - \operatorname{hinf} X_\bullet + 1$, then $\operatorname{Gpd}_{\mathcal{A}}M < \infty$.
\end{proposition}

\begin{proof}
	Let $t := \operatorname{hsup}X_\bullet$ and $s := \operatorname{hinf}X_\bullet.$ Then $t$ and $s$ are finite with $s \leqslant t.$ Since $X_\bullet$ is a finite quasi-Gorenstein projective resolution of $M$, we see that $X_\bullet \in \mathcal{C}^b(\mathcal{A})$, $\operatorname{H}_i(X_\bullet) \cong M^{\oplus a_i}$ for any $i\in [s,t]$ and $\operatorname{H}_i(X_\bullet) = 0$ for any $i > t$ or $i < s$, where $a_i$ are nonnegative integers and not all zero. Let $k=t-s$.
	Then $\operatorname{H}_i(X_\bullet[-s]) = \operatorname{H}_{i+s}(X_\bullet)=0$ for any $i > k$ or $i < 0$, and
	$\operatorname{H}_i(X_\bullet[-s]) \in \text{add}(M)$ for any $i\in[0,k]$.
	By Lemma~\ref{ext}, we have an isomorphism $X_\bullet[-s] \cong \bigoplus_{i=0}^k \operatorname{H}_i(X_\bullet[-s])[i]$ in $\mathcal D_{\mathrm{gp}}^-(\mathcal{A})$, which yields
	a series of isomorphisms
	$X_\bullet \cong \bigoplus_{i=0}^{k} \operatorname{H}_{i+s}(X_\bullet)[i+s] \cong \bigoplus_{j=s}^{t} \operatorname{H}_j(X_\bullet)[j]\cong\bigoplus_{j=s}^{t}M^{\oplus a_j}[j]$ in $\mathcal D_{\mathrm{gp}}^-(\mathcal{A}).$ 
	Now, Let $G_\bullet \rightarrow M \rightarrow 0$ be a proper Gorenstein projective resolution of $M.$ Then $M \cong G_\bullet$ in $\mathcal D_{\mathrm{gp}}^-(\mathcal{A}),$ and thus $X_\bullet\cong\bigoplus_{j=s}^{t}G_\bullet^{\oplus a_j}[j]$ in $\mathcal D_{\mathrm{gp}}^-(\mathcal{A}),$ where both sides belong to $\mathcal{K}^-(\mathcal{GP}).$ By \cite[Proposition 2.8]{GZ10}, 
	$\mathcal{K}^-(\mathcal{GP})$ is a triangulated subcategory of $\mathcal D_{\mathrm{gp}}^-(\mathcal{A})$, and it follows that
	$X_\bullet \cong \bigoplus_{j=s}^t G_\bullet^{\oplus a_j}[j]$ in $\mathcal{K}^-(\mathcal{GP}).$ Note that $\mathcal{K}^b(\mathcal{GP})$ is closed under direct summands in $\mathcal{K}^-(\mathcal{GP}).$ Since $X_\bullet \in \mathcal{K}^b(\mathcal{GP})$ and at least one \(a_i\) is not zero, we have $G_\bullet \in \mathcal{K}^b(\mathcal{GP}).$ This implies $\operatorname{Gpd}_\mathcal{A}M < \infty$.
\end{proof}

Proposition \ref{relation} implies the following result, which is the Gorenstein analogue of \cite[Corollary 3.4]{CCL26}.
\begin{corollary}\label{relation1}
	Let $M$ be an object of $\mathcal{A}$ admitting a proper Gorenstein projective resolution$.$ If $\operatorname{qGpd}_{\mathcal{A}}M< \infty$, and  $\operatorname{Ext}^n_{\mathcal{G}}(M, M) = 0$ for any $n\geqslant2$, then $\operatorname{Gpd}_{\mathcal{A}}M< \infty$.
\end{corollary}

Inspired by \cite[Theorem~1.2 (4)]{CCL26}, we consider the following proposition, which provides a method for computing the quasi-Gorenstein projective dimension.
	
\begin{proposition}\label{calculate}
	 Let $\cdots\rightarrow G_{n}\xrightarrow{d_n} G_{n-1}\rightarrow \cdots \rightarrow G_1 \xrightarrow{d_1} G_0\xrightarrow{d_0} M \rightarrow 0$ be a proper Gorenstein projective resolution of $M.$ Suppose that there is an integer $n \geqslant 2$ and a chain map:
       \[
		\begin{tikzcd}
			\cdots \arrow[r] 
			& G_{n+1} \arrow[r, "d_{n+1}"] \arrow[d, "g_{n+1}"] 
			& G_n \arrow[r, "d_{n}"] \arrow[d, "g_n"] 
			& G_{n-1} \arrow[r] \arrow[d, "0"] 
			& \cdots \arrow[r, "d_{1}"] 
			& G_0 \arrow[r] \arrow[d, "0"] 
			& 0 \\
			\cdots \arrow[r] 
			& G_1 \arrow[r, "d_{1}"] 
			& G_0 \arrow[r] 
			& 0 \arrow[r] 
			& \cdots \arrow[r] 
			& 0 \arrow[r] 
			& 0
		\end{tikzcd}
		\]
		inducing an isomorphism $\operatorname{Im}d_{n+m} \cong \operatorname{Im}d_{m}$ for some integer $m \geqslant 0$. Then $\operatorname{qGpd}_{\mathcal{A}} M \leqslant m$.
\end{proposition}
\begin{proof}
	Let $G_\bullet$ be the deleted complex $$\cdots\xrightarrow{d_{n+1}} G_{n}\xrightarrow{d_n} G_{n-1}\xrightarrow{d_{n-1}} \cdots \xrightarrow{d_2} G_1 \xrightarrow{d_1} G_0 \rightarrow 0.$$
	 For each $s\geqslant 0$, we denote by $G_{\bullet}^{\leqslant s}$ the brutal truncation at degree $s$, that is,
\begin{center}
		$G_i^{\leqslant s} = \begin{cases} 
		G_i & \text{if } i \le s; \\ 
		0 & \text{if } i > s. 
	\end{cases}$
\end{center}
Using a method similar to that in \cite[Theorem~1.2 (4)]{CCL26}, we have the following commutative diagram
	\[
	\begin{tikzcd}[column sep=8pt]
		\operatorname{Im}d_{n+m}[n+m-1] 
		\arrow[r] 
		\arrow[d, "\cong"] 
		& G_{\bullet}^{\leqslant n+m-1} 
		\arrow[r] 
		\arrow[d, "g_{\bullet}^{\leqslant n+m-1}"] 
		& M 
		\arrow[r] 
		\arrow[d] 
		& \operatorname{Im}d_{n+m}[n+m] 
		\arrow[d, "\cong"] \\
		\operatorname{Im}d_m[n+m-1] 
		\arrow[r] 
		& G_{\bullet}^{\leqslant m-1}[n] 
		\arrow[r] 
		\arrow[d] 
		& M[n] 
		\arrow[r] 
		\arrow[d] 
		& \operatorname{Im}d_m[n+m] \\
		& \operatorname{Cone}(g_{\bullet}^{\leqslant n+m-1}) 
		\arrow[r, equals] 
		\arrow[d] 
		& \operatorname{Cone}(g_{\bullet}^{\leqslant n+m-1}) 
		\arrow[d] \\
		& G_{\bullet}^{\leqslant n+m-1}[1] 
		\arrow[r] 
		& M[1],
	\end{tikzcd}
	\]
	 where all rows and columns are distinguished triangles
	 in $\mathcal D_{\mathrm{gp}}^b(\mathcal{A})$. Indeed, $M$ is $\mathcal {GP}$-quasi-isomorphic to the complex $$0 \rightarrow \operatorname{Im}d_{n+m}\rightarrow G_{n+m-1}\rightarrow \cdots \rightarrow G_1 \xrightarrow{d_1} G_0\rightarrow 0.$$ Therefore, the first row is a distinguished triangle
	 in $\mathcal D_{\mathrm{gp}}^b(\mathcal{A})$, and so is the second row. Moreover, the second column is a distinguished triangle
	 in $\mathcal K^b(\mathcal{A})$, which induces a distinguished triangle
	 in $\mathcal D_{\mathrm{gp}}^b(\mathcal{A})$, and the third column is constructed by the Octahedral axiom.  
	 
	 Since $M\rightarrow M[n] \rightarrow \operatorname{Cone}(g_{\bullet}^{\leqslant n+m-1}) \rightarrow M[1]$ is $\mathcal {GP}$-quasi-isomorphic (which is also quasi-isomorphic) to a distinguished triangle given by a mapping cone,
	 we have a long exact sequence of homologies
	  $$\cdots \rightarrow \operatorname{H}_i(M)\rightarrow \operatorname{H}_i(M[n]) \rightarrow \operatorname{H}_i(\operatorname{Cone}(g_{\bullet}^{\leqslant n+m-1})) \rightarrow \operatorname{H}_i(M[1])\rightarrow \cdots.$$ 
	 Since $n\geqslant 2$, we get \[
	\operatorname{H}_i(\operatorname{Cone}(g_\bullet^{\leqslant n+m-1})) \cong \begin{cases} 
		M, & i = 1 \text{ or } n, \\ 
		0, & \text{otherwise}. 
	\end{cases}
	\]
	 We now prove that the induced morphism $$l_i^{\operatorname{Cone}(g_\bullet^{\leqslant n+m-1})} \colon \operatorname{H}_i(\operatorname{Hom}_\mathcal{A}(G, \operatorname{Cone}(g_\bullet^{\leqslant n+m-1}))) \to \operatorname{Hom}_\mathcal{A}(G, \operatorname{H}_i(\operatorname{Cone}(g_\bullet^{\leqslant n+m-1})))$$ is an isomorphism
	 for any \(i \in \mathbb{Z}\) and any $G \in \mathcal{GP(A)}$, and then 
	 $\operatorname{Cone}(g_{\bullet}^{\leqslant n+m-1})$ is a quasi-Gorenstein projective resolution of $M.$ 
	 
	As $M\rightarrow M[n] \rightarrow \operatorname{Cone}(g_{\bullet}^{\leqslant n+m-1}) \rightarrow M[1]$ is $\mathcal {GP}$-quasi-isomorphic to a distinguished triangle given by a mapping cone, which is involved in a split exact sequence of complexes,  
	 we have a long exact sequence of homologies 
	 	$$\operatorname{H}_i(\operatorname{Hom}_\mathcal{A}(G,M[n])) \rightarrow \operatorname{H}_i(\operatorname{Hom}_\mathcal{A}(G,\operatorname{Cone}(g_{\bullet}^{\leqslant n+m-1}))) \rightarrow \operatorname{H}_{i}(\operatorname{Hom}_\mathcal{A}(G,M[1])).$$
	 	Therefore, we get  \[
	 	\operatorname{H}_i(\operatorname{Hom}_\mathcal{A}(G,\operatorname{Cone}(g_{\bullet}^{\leqslant n+m-1}))) \cong \begin{cases} 
	 		\operatorname{Hom}_\mathcal{A}(G,M), & i = 1 \text{ or } n, \\ 
	 		0, & \text{otherwise}. 
	 	\end{cases}
	 	\]
	 This shows that $l_i^{\operatorname{Cone}(g_\bullet^{\leqslant n+m-1})}$ is an isomorphism for any \(i \in \mathbb{Z}.\) Therefore, $\operatorname{Cone}(g_\bullet^{\leqslant n+m-1})$ is a finite quasi-Gorenstein projective resolution of $M.$ Consequently,
		$\operatorname{qGpd}_\mathcal{A}M \leqslant \sup(\operatorname{Cone}(g_\bullet^{\leqslant n+m-1})) - \operatorname{hsup}(\operatorname{Cone}(g_\bullet^{\leqslant n+m-1})) = n+m-n = m$.
\end{proof}

Next, we provide two examples to illustrate the computation of quasi-Gorenstein projective dimensions. The first example is due to Chen et. al \cite[Example 3.7]{CCL26}, which shows that there exists
a module $M$ with $\operatorname{Gpd}_{\mathcal{A}}M=\infty$ but $\operatorname{qGpd}_{\mathcal{A}}M=0.$ Thus, the inequality  $\operatorname{qGpd}_{\mathcal{A}}M \leqslant \operatorname{Gpd}_{\mathcal{A}}M$ can be strict.

\begin{example}\label{quiver1}
	\textnormal{Let $A$ be an algebra over a field $k$ presented by the quiver}
	\[\xymatrix{
		1 \ar[r]^{\alpha} & 2 \ar@(ur,dr)[]^{\beta}
	}\]
	{\rm with relations: \(\beta \alpha = 0 = \beta^2\). 
		Then it follows from \cite[Example 3.7]{CCL26} that $\operatorname{qpd}_A(S_1)= \infty$
		and $\operatorname{qpd}_A(S_2)= 0$.
Since $A$ is a monomial algebra,
	it follows from {\rm\cite[Theorem 4.1]{CSZ18}} that $\mathcal{GP}(A)=\mathcal{P}(A)$, and thus the quasi-Gorenstein projective dimension coincides with the quasi-projective dimension. Therefore,
	we have $\operatorname{qGpd}_A(S_1)= \infty$
	and $\operatorname{qGpd}_A(S_2)= 0$. However,  it is easy to see that
 $\operatorname{Gpd}_A(S_i)=\operatorname{pd}_A(S_i)=\infty$, for any $i=1$ or $i=2$.}
\end{example}

Next, we consider an algebra $A$ such that $\mathcal{GP}(A) \neq\mathcal{P}(A)$, and compute the quasi-Gorenstein projective dimension using our established results, rather than those of \cite{CCL26}. 

\begin{example}\label{quiver2}
{\rm	Let $A$ be an algebra over a field $k$ presented by the quiver
	\[
	\xymatrix{
		1\ar@(ul,dl)[]_{\alpha} \ar[r]^{\beta} & 2 \ar@(ur,dr)[]^{\gamma}
	}
	\]
	with relations: $ \alpha^2 =0= \gamma^2 = \gamma\beta.$ By \cite{CSZ18}, the only indecomposable non-projective Gorenstein projective $A$-module is $M=A\alpha.$ Then it follows from Theorem~\ref{dim} that $\operatorname{qGpd}_A(M) =\operatorname{Gpd}_A(M) = 0.$
	Moreover, it can be checked that
\[
\begin{tikzcd}[column sep=0.95cm]
	\cdots \arrow[r, "{\cdot \gamma}"] 
	& P_2 \arrow[r, "{\cdot \gamma}"] 
	& P_2 \arrow[r, "{\cdot \gamma}"] 
	& P_2 \arrow[r, "{\binom{\cdot \gamma}{0}}"] 
	& P_2 \oplus M \arrow[r, "{(\cdot \beta, l_M)}"] 
	& P_1 \arrow[r] 
	& S_1 \arrow[r] 
	& 0
\end{tikzcd}
\]
	is a proper Gorenstein projective resolution of $S_1$,
	where $l_M$ is the composition of $M\hookrightarrow \operatorname{rad}P_1$ and $ \operatorname{rad}P_1 \hookrightarrow P_1.$ This implies $\operatorname{Gpd}_AS_1= \infty$
	and $\operatorname{Ext}^n_{\mathcal{G}}(S_1, S_1) = 0$ for any $n\geqslant2.$ Then  $\operatorname{qGpd}_{A}S_1= \infty$ by Corollary~\ref{relation1}. 
	Also, there is a proper Gorenstein projective resolution of $S_2$ as follow:
\[
\begin{tikzcd}[column sep=0.95cm]
	P_\bullet = \cdots \arrow[r, "{\cdot \gamma}"] 
	& P_2 \arrow[r, "{\cdot \gamma}"] 
	& P_2 \arrow[r, "{\cdot \gamma}"] 
	& P_2 \arrow[r] 
	& S_2 \arrow[r] 
	& 0
\end{tikzcd}
\]
and thus $\operatorname{qGpd}_A(S_2) = 0$ by {\rm Theorem~\ref{basis} (4)}.}
\end{example}

\noindent {\footnotesize {\bf ACKNOWLEDGMENT.}
This work is supported by
the National Natural Science Foundation of China (12561008),
the project of Young and Middle-aged Academic and Technological leader of Yunnan
(202305AC160005) and the Basic Research Program of Yunnan Province (202301AT070070).}

\end{document}